\documentclass[reqno]{amsart}
\usepackage{graphicx}
\usepackage{caption}
\usepackage{subcaption}
\usepackage{color}

\usepackage{amssymb}
\usepackage{mathrsfs}
\usepackage{enumitem}
\usepackage{hyperref}
\usepackage{isomath}

\begin{document}

\newtheorem{thm}{Theorem}[section]
\newtheorem{lem}[thm]{Lemma}
\newtheorem{cor}[thm]{Corollary}
\newtheorem{add}[thm]{Addendum}
\newtheorem{prop}[thm]{Proposition}
\theoremstyle{definition}
\newtheorem{defn}[thm]{Definition}
\theoremstyle{remark}
\newtheorem{rmk}[thm]{Remark}

\newcommand{\CC}{\mathbb{C}}
\newcommand{\RR}{\mathbb{R}}
\newcommand{\Z}{\mathbb{Z}}
\newcommand{\QQ}{\mathbb{Q}}
\newcommand{\NN}{\mathbb{N}}
\newcommand{\CmodTwoPiIZ}{{\mathbf C}/2\pi i {\mathbf Z}}
\newcommand{\Cnozero}{{\mathbf C}\backslash \{0\}}
\newcommand{\Cinfty}{{\mathbf C}_{\infty}}
\newcommand{\RPnminustwo}{\mathbb{RP}^{n-2}}

\newcommand{\SLtwoC}{\mathrm{SL}(2,{\mathbb C})}
\newcommand{\SLtwoR}{\mathrm{SL}(2,{\mathbb R})}
\newcommand{\PSLtwoC}{\mathrm{PSL}(2,{\mathbb C})}
\newcommand{\PSLtwoR}{\mathrm{PSL}(2,{\mathbb R})}
\newcommand{\SLtwoZ}{\mathrm{SL}(2,{\mathbb Z})}
\newcommand{\PSLtwoZ}{\mathrm{PSL}(2,{\mathbb Z})}

\newcommand{\A}{{\mathcal A}}
\newcommand{\B}{{\mathcal B}}
\newcommand{\C}{{\mathcal X}}
\newcommand{\D}{{\mathcal D}}
\newcommand{\E}{{\mathcal E}}

\newcommand{\MCG}{\mathcal{MCG}}
\newcommand{\HH}{H^2}
\newcommand{\HHH}{H^3}
\newcommand{\tr}{{\hbox{tr}\,}}
\newcommand{\Hom}{\mathrm{Hom}}
\newcommand{\SL}{\mathrm{SL}}
\newcommand{\BQ}{\rm{BQ}}
\newcommand{\Id}{\rm{Id}}

\newcommand{\setn}{{[n]}}
\newcommand{\powern}{{P(n)}}
\newcommand{\nck}{{C(n,k)}}

\newcommand{\hatI}{{\hat{I}}}
\newcommand{\TkDelta}{{T^{|k|}(\Delta)}}
\newcommand{\vecDelta}{{\vec{\Delta}_{\phi}}}

\newcommand{\Tabstwo}{{T^{|2|}(\Delta)}}
\newcommand{\Tnminusone}{{T^{|n-1|}(\Delta)}}
\newcommand{\Hur}{{\mathcal{H}}}

\newenvironment{pf}{\noindent {\it Proof.}\quad}{\square \vskip 10pt}

\title[Automorphisms preserving the Markoff-Hurwitz polynomial]{Polynomial automorphisms of $\mathbb{C}^n$ preserving the Markoff-Hurwitz polynomial}
\author[H. Hu, S.P. Tan, and Y. Zhang]{Hengnan Hu, Ser Peow Tan, and Ying Zhang}
\address{Department of Mathematics \\ National University of Singapore \\ Singapore 119076} \email{huhengnan@gmail.com}
\address{Department of Mathematics \\ National University of Singapore \\ Singapore 119076} \email{mattansp@nus.edu.sg}
\address{School of Mathematical Sciences \\ Soochow University \\ Suzhou 215006 \\ China} \email{yzhang@suda.edu.cn}


\thanks{Hu and Tan are partially supported by the National University of Singapore academic research grant R-146-000-186-112. Zhang is supported by NSFC (China) grant no. 11271276 and Ph.D. Programs Foundation (China) grant no. 20133201110001.
}

\dedicatory{To Professor William Goldman on the occasion of his sixtieth birthday}

%
%

 \begin{abstract}
 We study the action of the group of polynomial automorphisms of $\CC^n$ ($n \ge 3$) which preserve the Markoff-Hurwitz polynomial
 $$ H(\mathbfit{x}):=x_1^2+x_2^2+\cdots +x_n^2- x_1 x_2 \cdots x_n. $$
 Our main results include the determination of the group, the description of a non-empty open subset of $\CC^n$ on which
 the group acts properly discontinuously (domain of discontinuity), and identities for the orbit of points in the domain of discontinuity.
 \end{abstract}

 \maketitle
 \tableofcontents

 \vspace{10pt}


 \section{Introduction}\label{s:intro}

 The Markoff equation
 \begin{equation}\label{eqn:Markoff}
 x_1^2+x_2^2+x_3^2-x_1x_2x_3=0
 \end{equation}
 occurs in various settings and has been well studied by many authors in different contexts (actually Markoff studied the equation $a^2+b^2+c^2=3abc$, the triples of positive integer solutions are called Markoff triples, but this is equivalent to (\ref{eqn:Markoff}) by the change of variables $x_1=3a, x_2=3b, x_3=3c$). As a diophantine equation, the triples of positive integer solutions have interpretations in terms of diophantine approximations, minima of binary quadratic forms and traces of simple closed geodesics on the modular torus, see \cite{series1985mi} for an excellent survey. A basic theorem of Markoff states that the set of (non-trivial) integer solutions of (\ref{eqn:Markoff}) can be obtained from the fundamental solution $(3,3,3)$ by considering its orbit under the action of the group $\Gamma_3^*$ of polynomial automorphisms of (\ref{eqn:Markoff}). Here $\Gamma_3^*$ is generated by the permutations, the even sign change automorphisms (where two of the variables change signs) and the involution $b_1:\CC^3 \rightarrow \CC^3$ given by
 $$ b_1(x_1,x_2,x_3)=(x_2x_3-x_1,x_2, x_3). $$

 More generally, $\CC^3$ parametrizes
 $$ {\rm Hom}(F_2, \SLtwoC)/\!/\SLtwoC, $$
 the  character variety of the free group on two generators $F_2=\langle A,B \rangle$ into $\SLtwoC $, where the parameters are $x_1=\tr A$, $x_2=\tr B$ and $x_3=\tr (AB)$. The variety described by (\ref{eqn:Markoff}) parametrizes the relative character variety
 $$ \{[\rho] \in {\rm Hom}(F_2, \SLtwoC)/\!/\SLtwoC ~|~\tr (ABA^{-1}B^{-1})=-2\}, $$
 see \cite{bowditch1998plms,goldman2003gt, tan-wong-zhang2006gd,tan-wong-zhang2008advm}. In this context,  $\Gamma_3^*$ is commensurable to the action of $\rm{Out}(F_2)$ on the (relative) character variety, and the basic theorem of Markoff can be interpreted as saying that the set of positive Markoff triples is the orbit of the root triple $(3,3,3)$ under the action of  $\rm{Out}(F_2)$ (Cohn \cite{Cohn} was the first to notice the connection between the Markoff triples and traces of simple geodesics on the modular torus).

 Various subsets of the character variety have interpretations as Fricke spaces of geometric structures (for example the real points on (\ref{eqn:Markoff}) except the origin correspond to hyperbolic structures on a once-punctured torus), and $\rm{Out}(F_2)$ acts properly discontinuously on these subsets. On the other hand, other subsets have interpretations as representations of $F_2$ into $\rm{SU}(2)$ and results of Goldman \cite{goldman1997annals, goldman2003gt} imply that $\rm{Out}(F_2)$ acts ergodically on these subsets. Hence, the overall action of $\Gamma_3^*$ on $\CC^3$ is dynamically interesting and quite mysterious.

 \medskip

 Hurwitz \cite{hurwitz1907} generalized the study to the $n$-variable diophantine equation
 \begin{equation}\label{eqn:Markoff-Hurwitz1}
 x_1^2+x_2^2+\cdots +x_n^2-k x_1x_2\cdots x_n=0,
 \end{equation}
 (where $k \in \mathbb Z$), and obtained analogous results to those for the Markoff equation. In particular, he showed that there are a finite set of basic solutions of (\ref{eqn:Markoff-Hurwitz1}) such that all integral  solutions  can be obtained by considering the orbits of these basic solutions under a certain group action. By a simple change of variables by a homothety $\mathbfit{x} \mapsto \lambda \mathbfit{x}$, one may assume that $k=1$ in the above, and consider the Markoff-Hurwitz polynomial:
 \begin{equation}\label{eqn:Markoff-Hurwitz}
 H(x_1, x_2 \cdots, x_n):=x_1^2+x_2^2+\cdots +x_n^2-x_1x_2 \cdots x_n
 \end{equation}
 ($n=3$ corresponds to the Markoff polynomial on the left hand side of (\ref{eqn:Markoff})).

 In this context, it is interesting to determine the group of polynomial automorphisms of $\CC^n$ which preserves (\ref{eqn:Markoff-Hurwitz}), and more generally to study the dynamics of this group action on $\CC^n$. We state here three general questions which we address in this paper:

 Question 1. What is  the group of polynomial automorphisms of $\CC^n$ preserving the Markoff-Hurwitz polynomial (\ref{eqn:Markoff-Hurwitz})?

 Question 2. How can we describe invariant open subsets of $\CC^n$  on which this group acts properly discontinuously (domains of discontinuity)? These sets would be analogues of Teichm\"uller or Fricke spaces, and the quotient under the group action the analogues of various moduli spaces.

 Question 3. What can we say about the orbits of points in the domain of discontinuity, in particular, what is their growth rate and do they satisfy some  identities? This question is motivated by results of McShane \cite{mcshane1991thesis} who proved a remarkable identity for once-punctured hyperbolic tori which can be interpreted as an identity satisfied by the Markoff triples, and more generally, triples of real numbers satisfying (\ref{eqn:Markoff}). This was generalized by Bowditch \cite{bowditch1998plms}, Tan-Wong-Zhang \cite{tan-wong-zhang2008advm} and Hu-Tan-Zhang \cite{hu-tan-zhang2014mrl} in the context of the group action on $\CC^3$ preserving (\ref{eqn:Markoff}). The natural question here is whether these identities generalize for $\CC^n$,  $n \ge 4$.

 The main purpose of this paper is to give answers to the above questions.
 For Question 1, we have the following (compare to \cite{horowitz1975tams} which considers the case $n=3$):

 \begin{thm}\label{thm:polyauto}
 The group of polynomial automorphisms of $\CC^n$ preserving the Markoff-Hurwitz polynomial
 \begin{equation}\label{eqn:Hurwitz}
 H(\mathbfit{x})=H(x_1, x_2,\cdots, x_n):=x_1^2+x_2^2+\cdots +x_n^2-x_1x_2 \cdots x_n
 \end{equation}
 is given by $$\Gamma_n^* =\Gamma_n \rtimes \Lambda$$
 where
 \begin{enumerate}
 \item $\Lambda$ is the group of linear automorphisms of $(\CC^n, H(\mathbfit{x}))$ and is given by  $\Upsilon_n \rtimes S_n$ where $\Upsilon_n$ is the group of even sign change automorphisms and  $S_n$ the symmetric group acting by permutations on the coordinates; and
 \item  $\Gamma_n=\langle\, b_1, \cdots, b_n ~| ~b_i^2={\Id},i=1,\cdots,n\, \rangle$ where $b_i$, $i=1, \cdots, n$ is an involution of $\CC^n$ which fixes $x_j$ for $j \neq i$ and replaces $x_i$ with the product of the other coordinates subtract $x_i$, that is,
 \[
 b_i (x_1, \cdots, x_i,\cdots, x_n) := \Big( x_1,\cdots,x_{i-1}, \prod_{j \neq i}x_j-x_i, x_{i+1},\cdots,x_n \Big).
 \]
 \end{enumerate}
 In particular, $\Gamma_n$ is a normal subgroup of finite index in $\Gamma_n^*$.
 \end{thm}

 \medskip

 For the second question, it is convenient, and sufficient, to consider  the action of $\Gamma_n$ on $\CC^n$. From a dynamical point of view, this is equivalent to studying the action of $\Gamma_n^*$ on $\CC^n$, since $\Gamma_n$ is a normal subgroup of finite index in $\Gamma_n^*$.  The group $\Gamma_n$ is easier to handle as it is a Coxeter group generated by $n$ involutions with no other relations, and the Cayley graph $\Delta$ of $\Gamma_n$ is a  regular (edge-labeled), rooted $n$-valence tree.  The vertex set $V(\Delta)$ consists of the elements of $\Gamma_n$ (with root the identity element ${\Id} \in \Gamma_n$) and the edge set $E(\Delta)$ is labeled by $1,2, \cdots, n$; two vertices $g, h \in V(\Delta)$ are connected by an $i$-edge if and only if $g=h \cdot b_i$.
 For any subtree $\Delta' \subset \Delta$, denote by $V(\Delta')$ and $E(\Delta')$ the vertex set and edge set of $\Delta'$ respectively, note that $E(\Delta')$ has a labeling induced from the labeling on $E(\Delta)$.

 \medskip

 An element $\mathbfit{a}=(a_1, \cdots, a_n) \in \CC^n$ induces a map
 $$ \Phi:=\Phi_{\mathbfit{a}}:V(\Delta) \rightarrow \CC^n,  ~~\hbox{where}~~\Phi(g)= g^{-1}(\mathbfit{a}), $$
 which satisfies the following edge relations:

 \begin{defn}[Edge relations] Suppose that $g, h \in V(\Delta)$ are connected by an $i$-edge, and $\Phi(g)=(x_1, \cdots, x_i, \cdots, x_n)$ and $\Phi(h)=(y_1, \cdots, y_i, \cdots, y_n)$. Then
 \begin{eqnarray}\label{eqn:edgerelation}
 \nonumber  x_j &=& y_j, \quad j\neq i, \quad 1 \le j \le n  \\
  x_i+y_i &=& x_1\cdots\hat x_i \cdots x_n= y_1\cdots\hat y_i \cdots y_n
 \end{eqnarray}
 \end{defn}

 \begin{figure}[ht]
  \def\svgwidth{\columnwidth}
  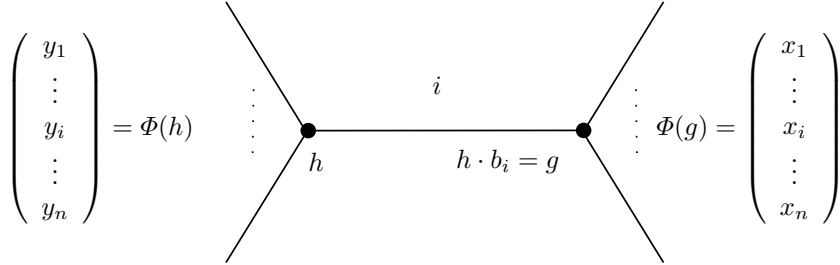
  \caption{Edge relation}
  \label{fig:edegrelation}
 \end{figure}

 We call maps from $V(\Delta)$ to $\CC^n$ which satisfy the edge relations {\it vector-valued Hurwitz} maps. The vector-valued Hurwitz map $\Phi$ describes the orbit of $\mathbfit{a}$ under $\Gamma_n$. Furthermore, because of the edge relations, $\Phi$ is determined by its value on any vertex, in particular, the root; so the set of vector-valued Hurwitz maps can be identified with $\CC^n$.

 \medskip

 We next describe an important set of geometric objects, the alternating geodesics (or 2-regular subtrees of $\Delta$), which will play an important role in describing the domain of discontinuity, as well as the identities. These are bi-infinite geodesics $\gamma \subset \Delta$ consisting of alternating $i$, $j$-edges, where $1 \le i <j \le n$. Denote by $\A$ the set of alternating geodesics and $\A_{ij}$ the set of alternating geodesics consisting of $i$-, $j$-edges. Then $\A = \bigsqcup \A_{ij}$, where the union is taken over all distinct pairs $i,j$. The set $\A$ has some very interesting geometric properties which are explored in greater detail in \S \ref{ss:subtrees} and \S \ref{ss:Fibonacci}. For the purpose of studying the dynamics of the group action, we first note that a vector-valued Hurwitz map $\Phi: V(\Delta) \rightarrow \CC^n$ induces a  map $\phi: \A \rightarrow \CC$ as follows: Given an $\{i,j\}$-alternating geodesic $\gamma$, we note that for all vertices $g \in V(\gamma)$ and $k \neq i,j$, the $k$-th entries of $\Phi(g)$ are the same since by the edge relations (\ref{eqn:edgerelation}) only the $i$th and $j$th entries change when moving along the edges of the alternating geodesic.

 \begin{defn}\label{def:twoinducedmaps}
 Suppose $\gamma \in \A_{ij}$ and $g \in V(\gamma)$ with $\Phi(g)=(x_1, x_2, \cdots, x_n)$. Then
 the {\it (extended) Hurwitz map} $\phi: \A \rightarrow \CC$ is given by $${\displaystyle \phi(\gamma)=\prod_{k \neq i,j} x_k.}$$
 There is also a secondary map, the {\it square sum weight} $\sigma: \A \rightarrow \CC$, given by
 $$ {\displaystyle \sigma(\gamma)=\sum_{k \neq i,j} x_k^2} $$
 which will be useful for the identities later.
 \end{defn}

 \medskip
 Fix $\mathbfit{a} \in \CC^n$, let $\Phi:V(\Delta) \rightarrow \CC^n$ be the induced vector-valued Hurwitz map,  $\phi: \A \rightarrow \CC$ the induced (extended) Hurwitz map and let $K>0$. Define
 $$ \A_{\phi}(K)=\{ \gamma \in \A ~|~ |\phi(\gamma)| \le K\}. $$

 Our second main theorem answering Question 2 is as follows:

 \begin{thm}\label{thm:domainofdis}
 Let $\D \subset \CC^n$ be the set of $\mathbfit{a} \in \CC^n$ for which the induced (extended) Hurwitz map $\phi:=\phi_{\mathbfit{a}}: \A \rightarrow \CC$ satisfy the conditions
 \begin{enumerate}
 \item[(i)] $\phi(\gamma) \not\in [-2,2]$ for all $\gamma \in \A$;
 \item[(ii)] the set $\A_{\phi}(K)$ is finite for some $K>2$.
 \end{enumerate}
 Then $\D$ is a non-empty open subset of $\CC^n$ invariant under $\Gamma_n$ and $\Gamma_n$ acts properly discontinuously on $\D$. Furthermore, $\mathbf{0} \in \CC^n$ does not lie in the closure of $\D$.
 \end{thm}

 Although the sets $\D$  above are defined by rather simple conditions, they are typically very complicated with fractal like boundary. Figure \ref{fig:diagslices} shows computer generated images of the {\it diagonal slices} of the sets (where $\mathbfit{a}=(a,a,\cdots, a) \in \D$) in the case where $n=3$ and $4$ (points in the black part correspond to points in $\D$).
 Furthermore, in the $n=3$ case, the diagonal slice is very different from the diagonal slice of the Schottky space, a set which is defined from geometric considerations, see \cite{STY}.

 \begin{figure}[ht]
\centering
\begin{subfigure}{.5\textwidth}
  \centering
  \includegraphics[width=.85\linewidth]{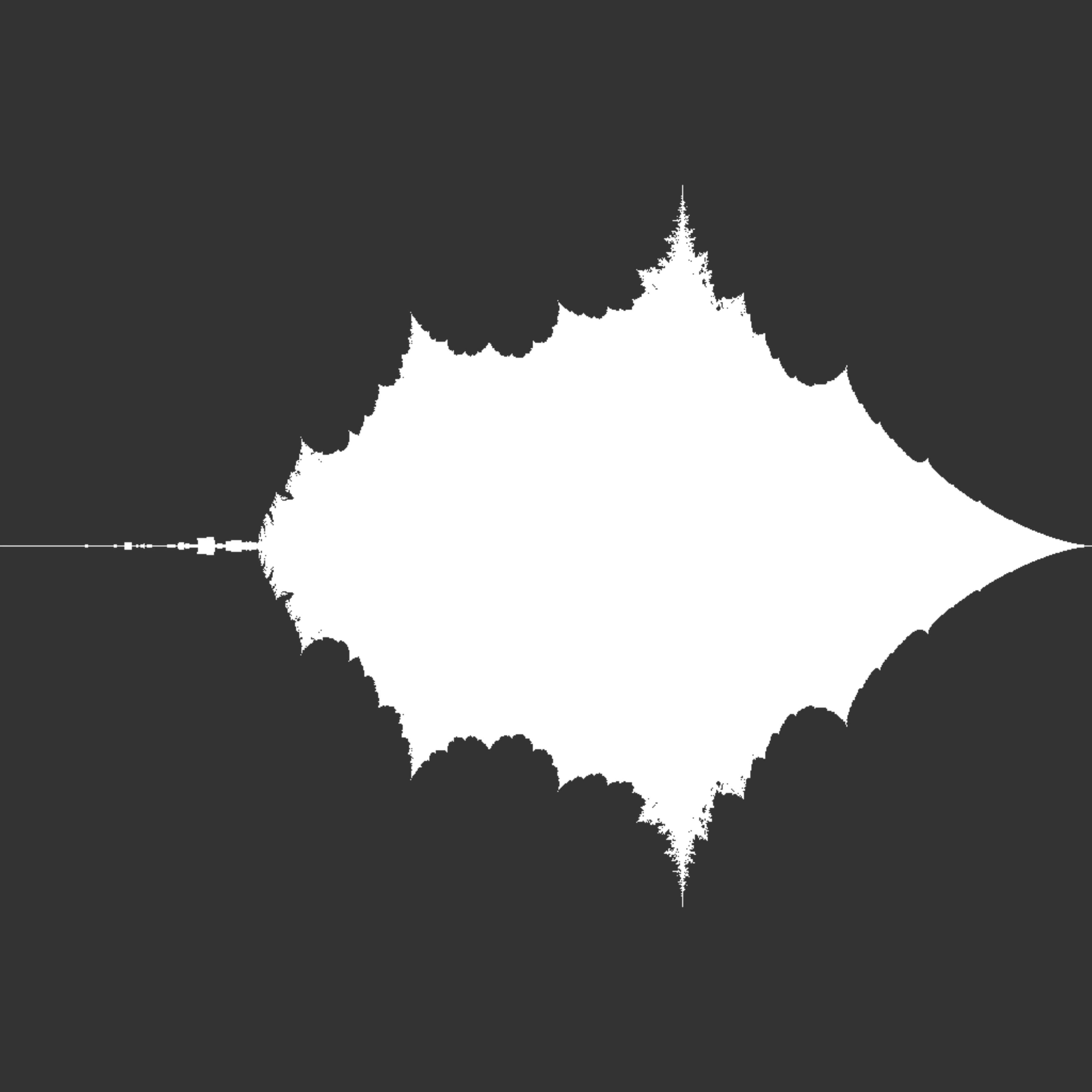}
\end{subfigure}%
\begin{subfigure}{.5\textwidth}
  \centering
  \includegraphics[width=.85\linewidth]{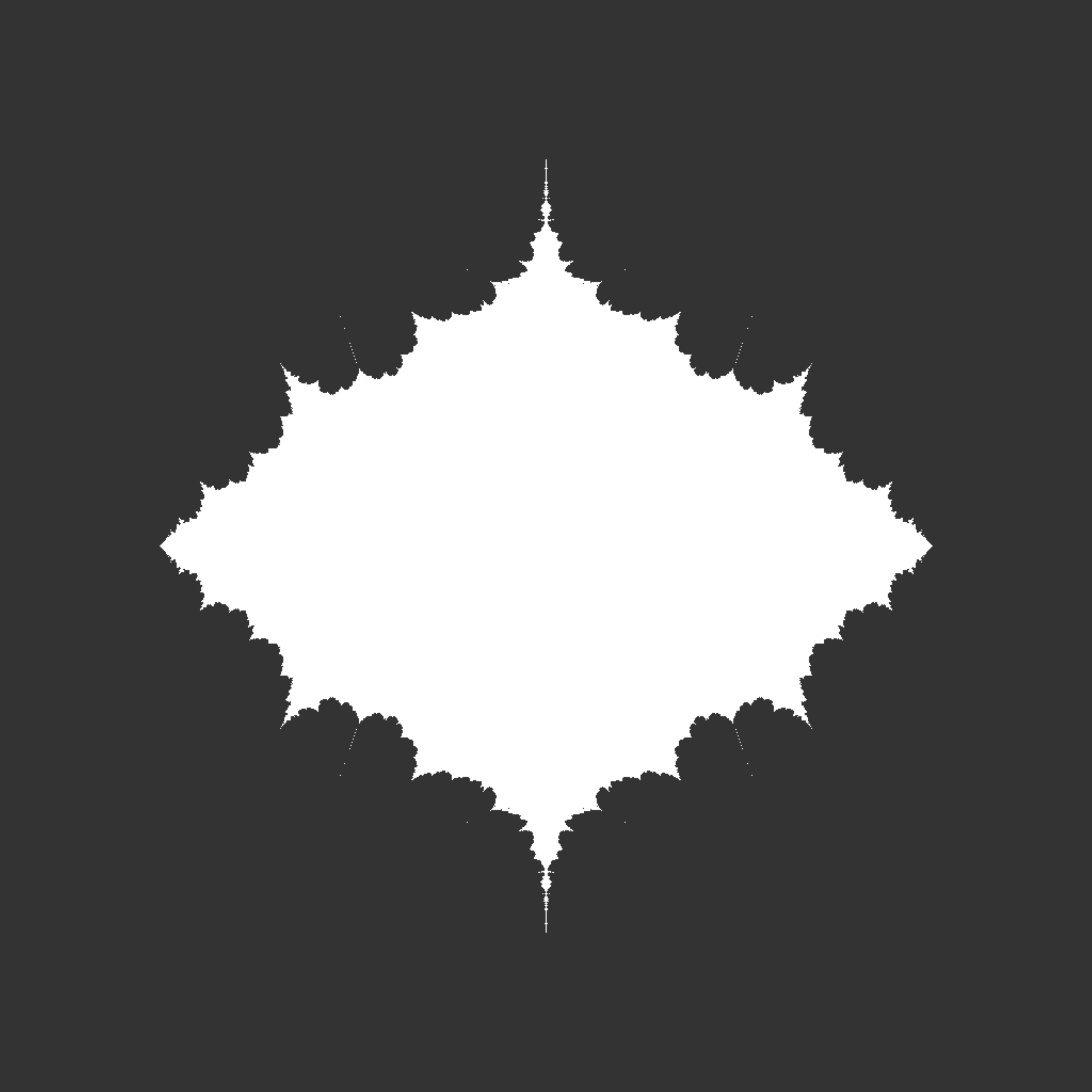}
\end{subfigure}
\caption{Diagonal slices of the set $\D$ for the cases $n=3$ (left) and $n=4$, Courtesy of Yasushi Yamashita. }
\label{fig:diagslices}
\end{figure}

 An important property of the sets $\A_{\phi}(K)$ for $K \ge 2$ is that they are {\it edge-connected} (see Proposition \ref{prop:e-c}). That is, if $\gamma, \gamma' \in \A_{\phi}(K)$, there exists a sequence $\gamma_1, \cdots, \gamma_m \in \A_{\phi}(K)$ such that $\gamma=\gamma_1$, $\gamma'=\gamma_m$ and $\gamma_i$ and $\gamma_{i+1}$ share an edge for $i=1, \cdots, m-1$. This makes it possible to write a computer program to search for the set $\A_{\phi}(K)$ when it is finite, and hence determine the diagonal slice of  $\D$.

 Another important property of  the maps $\phi$ induced by elements of $\D$ is that the function $\log ^+|\phi|=\max \{0, \log |\phi|\}$ has Fibonacci growth (see \S \ref{ss:Fibonacci} for the definition of the Fibonacci function on $\A$ and Propositions \ref{prop:upperfibonacci} and \ref{prop:lowerfibonacci} for the precise statement). This  allows us to deduce the absolute convergence of certain sums over $\A$, which are used in the proofs of the last set of results concerning identites for orbits of elements in $\D$. More precisely, we have the following theorem which answers Question 3.

 \begin{thm}\label{thm:identity}
 Suppose that $\mathbfit{a}=(a_1, \cdots, a_n) \in \D \subset \CC^n$ as in Theorem \ref{thm:domainofdis}. Suppose further that
 $$ a_1^2+a_2^2+\cdots +a_n^2-a_1a_2\cdots a_n=\mu. $$
 Then
 \[
 \sum_{\gamma \in \A } \mathfrak{h}_{\mu}(\gamma) = 1
 \]
 where the series converges absolutely and
 \[
 \mathfrak{h}_{\mu}(\gamma):= 1- \left(1+\frac{2 \mu}{n(n-1) (\sigma(\gamma)-\mu)}\right) \sqrt{1-\frac{4}{(\phi(\gamma))^2}},
 \]
 where $\phi, \sigma: \A \rightarrow \CC$ are the (extended) Hurwitz map and  the square sum weight given in Definition \ref{def:twoinducedmaps}.

 In particular, when $\mu=0$ we have
 \[
 \sum_{\gamma \in \A} h(\phi(\gamma)) = 1
 \]
 where the sum converges absolutely and for $z \neq 0$
 \[
 h(z)=1-\sqrt{1-\frac{4}{z^2}}.
 \]
 \end{thm}
%

 \noindent {\it Remarks:}
 \begin{enumerate}
 \item The case where $n=3$, $(a_1,a_2,a_3) \in \RR^3 \setminus \{(0,0,0)\}$ and $\mu=0$ is the original McShane identity, the case where $(a_1,a_2,a_3) \in \CC^3$,  $\phi$ satisfies the conditions of Theorem \ref{thm:domainofdis} and $\mu=0$ is Bowditch's generalization, and when $\mu \neq 0$, this was done by the authors in \cite{hu-tan-zhang2014mrl}.

 \item For the case where $n=4$ and $\mu=0$ Huang and Norbury \cite{HN} found a very interesting geometric interpretation as an identity for simple one-sided geodesics on the thrice-punctured projective plane.


 \item There are other ways in which the group $\Gamma_n$ can act on $\CC^n$ which are of interest geometrically. An example of this can be found in \cite{maloni-palesi-tan2013ggd} which studies the action of $\Gamma_3$ on $\CC^3$ arising from the action of the mapping class group on the relative character variety of a four-holed sphere, fixing the traces of the four boundary components. Another interesting example can be obtained by looking at the relative character variety of the four-holed sphere where we fix the traces of three interior simple closed curves to be zero and look at the action of a different subgroup of $\rm{Out}(F_3)$ on this, preserving the traces of the interior curves. In this case we get an action of $\Gamma_4$ on $\CC^4$ which is similar to the one we consider. In these cases, the polynomials preserved by the group action have degree $n$ and the methods we use here can be applied to these cases directly.

 \item Another very interesting example of actions of $\Gamma_n$ on $\RR^n$ occurs in the study of Apollonian circle/sphere packings, here the group action preserves a quadratic polynomial and it would be interesting to try to apply the techniques developed here to this class of problems.

 \item An earlier version of some of the results here are contained in the first author's Ph.D. thesis \cite{hu2013thesis}.

 \end{enumerate}

 \bigskip

 The rest of this paper is organized as follows. We defer the proof of Theorem \ref{thm:polyauto} to the Appendix as the proof uses elementary techniques which are somewhat independent from the rest of the paper. In \S \ref{s:setting}, we describe the combinatorial setting for the problem,  set up the notation, introduce some of the basic geometric and combinatorial objects, and prove some basic results. \S \ref{s:domainofdis} is one of the main parts of the paper where we prove some of the main results for Hurwitz maps, leading up to the proof of  Theorem \ref{thm:domainofdis}. In \S \ref{s:proof} we discuss the identities and prove Theorem \ref{thm:identity}. Finally, in \S \ref{s:conclusion}, we discuss some open problems and directions for further study.

 \bigskip

 \noindent {\it Acknowledgements}. We are grateful to Martin Bridgeman, Dick Canary, Fran\c{c}ois Labourie, Bill Goldman, Makoto Sakuma and Yasushi Yamashita for helpful conversations and comments. In particular, we are very grateful to Dick Canary who had asked about the domain of continuity for the action of $\Gamma_n$ on $\CC^n$, to Yasushi Yamashita who had written the computer program which drew the beautiful slices in Figure \ref{fig:diagslices}, and to Bill Goldman for his constant support and encouragement. We would also like to thank Yi Huang and Paul Norbury for bringing our attention to \cite{HN} and the geometric interpretation of our identity for the $n=4$ homogeneous case, see Remark (2) above.


 \section{Combinatorial setting and basic results}\label{s:setting}

 \subsection{Cayley graphs and orbits }\label{ss:cayley}
 Let $\Gamma_n:=\langle\, b_1, \cdots, b_n \mid b_1^2= \cdots= b_n^2={\mathrm{Id}} \,\rangle$ act on $\CC^n$ as follows:
 \[
 b_i (x_1, \cdots, x_i, \cdots, x_n) := \Big(x_1,\cdots,x_{i-1},\prod_{j \neq i}x_j-x_i,x_{i+1},\cdots,x_n\Big).
 \]
 Clearly, the Markoff-Hurwitz polynomial (\ref{eqn:Hurwitz}) is invariant under $\Gamma_n$.


 The Cayley graph $\Delta:=\Delta(\Gamma_n)$ of $\Gamma_n$ is a {\em rooted, regular} $n$-tree with vertex set $V(\Delta)=\Gamma_n$, where the root $v_0$ corresponds to the identity element of $\Gamma$. Two vertices $g, h \in V(\Delta)$ are connected by an edge if and only if $g=hb_i$ for some $i \in \{1,2, \cdots, n\}$. The corresponding edge of $\Delta$ is labeled (or colored) by $i $, so the edges incident to any given vertex are labeled by the distinct elements of $\{1,\cdots,n\}$, giving a regularly labeled $n$-tree. See Figure \ref{fig:4rootedtree} for the case of $n=4$.

\begin{figure}[ht]
\def\svgwidth{.8\columnwidth}
  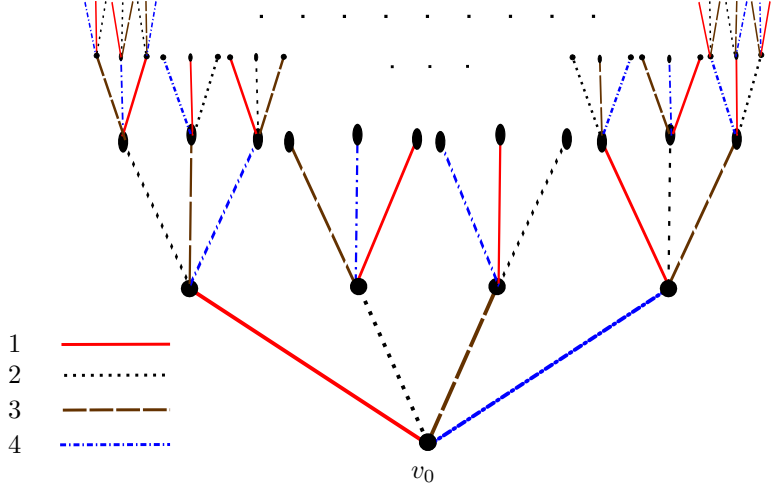
  \caption{The rooted regular tree $\Delta$ for $n=4$}
  \label{fig:4rootedtree}
\end{figure}

An element $\mathbfit{a} \in \CC^n$ induces a {\em vector-valued  Hurwitz map} 
$$ \Phi:=\Phi_{\mathbfit{a}}: V(\Delta) \rightarrow \CC^n $$ 
where $\Phi(g)=g^{-1}(\mathbfit{a})$. $\Phi$ satisfies the edge relations (\ref{eqn:edgerelation}). Furthermore, the set $\{  \mathbfit{x} =\Phi(g) ~|~ g \in \Gamma_n\}$ is the orbit of $\mathbfit{a}$ under $\Gamma_n$.

There are many interesting combinatorial objects and structures associated to $\Delta$, and  maps from these objects to $\CC$ induced from $\Phi$ which we describe in the next sections. We  will emphasize the combinatorial setting. The group action will be reflected by various edge relations satisfied by the maps.

\subsection{Index sets $\setn$, $\powern$, $\hatI$, $\{\hat{i}\}$ and $\nck$}\label{ss:index}
We set some notation for index sets. Let $\setn:=\{1,2, \cdots, n\}$ and $\powern$ the power set of $\setn$. For $I \in \powern$, $|I|$ denotes the cardinality of $I$ and $\hatI:=\setn \setminus I$ denotes the complement of $I$ in $\setn$. Let $\nck =\{I \in  \powern ~:~ \vert I \vert =k\} $. For $I=\{i\}$,  use $\{\hat{i}\}$ to denote $\hatI$.

\subsection{Rooted regular trees}\label{ss:trees}
Let $\Delta$ be a {\em rooted, regularly labeled} $n$-tree, that is, an $n$-valent tree with a distinguished vertex $v_0$, the root, such that incident to each vertex, the $n$ edges are labeled by distinct elements of $\setn$.

We put the standard metric on $\Delta$ where every edge has length one and denote by $d(T)$ the distance of any subtree $T \subset \Delta$ to $v_0$.

Denote by $V(\Delta)$ and $E(\Delta)$  the vertex set and the edge set of $\Delta$ respectively. For $e \in E(\Delta)$ denote by $\bar{e}$ the closure of $e$, that is, the subtree consisting of $e$ and the two vertices incident to $e$. Every $e \in E(\Delta)$ can be directed in two ways, denote by $\vec{E}(\Delta)$ the set of directed edges of $\Delta$ and for $\vec{e} \in \vec{E}(\Delta)$ denote by $-\vec{e}$ the directed edge with the same underlying edge $e$ but directed in the opposite direction of $\vec{e}$.
We think of a directed edge $\vec{e}$  as specifying the two vertices incident to $e$ as the {\em tail} and {\em head} (we speak of $\vec{e}$ as being `directed towards' its head).  Note that every directed edge $\vec{e}$ is either directed towards the root $v_0$, or away from it.

%
%
%

\subsection{Regular subtrees: $\TkDelta$, $\C:=T^{|n-1|}(\Delta)$ and $\A:=T^{|2|}(\Delta)$}\label{ss:subtrees}
Let $\Delta$ be a rooted regularly labeled $n$-tree and let $I \in \nck$. Denote by $E_I(\Delta) \subset E(\Delta)$ the set of edges labeled by the elements of $I$; in particular $E(\Delta) = \bigsqcup_{i=1}^n E_{\{i\}}(\Delta)$. Then $\Delta \setminus E_{\hatI}(\Delta)$ is a disjoint union of $I$-regular trees, we call the set of connected components of $\Delta \setminus E_{\hatI}$ the $I$-forest of $\Delta$, denoted by $T^I(\Delta)$; its elements are exactly the regularly labeled $k$-subtrees of $\Delta$ labeled by $I$. We denote the set of all regularly labeled $k$-subtrees of $\Delta$ by $\TkDelta$, so
\[
\TkDelta=\bigsqcup_{I \in \nck} T^I(\Delta).
\]
In particular, we have the following decomposition of $T^{|n-1|}(\Delta)$:
\[T^{|n-1|}(\Delta)=\bigsqcup_{i=1}^{n}T^{\{\hat{i}\}}(\Delta).\]


Thus $T^{|0|}(\Delta)=V(\Delta)$, $T^{|1|}(\Delta)=\overline{E}(\Delta):=\{\bar{e} ~|~ e \in E(\Delta)\}$,
$T^{|2|}(\Delta)$ is the set of alternating geodesics introduced in \S\ref{s:intro}, and $T^{|n|}(\Delta)=\{\Delta\}$. We shall be particularly interested in  $T^{|2|}(\Delta)$ and $T^{\{i,j\}}(\Delta)$, $1 \le i <j \le n$, which we denote by $\A$ and $\A_{ij}$, these are the set of alternating geodesics described in the introduction. In addition, $\Tnminusone$ is denoted by $\C$, and $T^{\{\hat{i}\}}(\Delta)$ by $\C_i$.
See Figure \ref{fig:elementsofCD} for examples of the case $n=4$. We will be mostly interested in maps from $\C$ and $\A$ into $\CC$ induced by a vector-valued Hurwitz map $\Phi$.

\begin{figure}[ht]
\begin{subfigure}{.3\textwidth}
  \centering
  \def\svgwidth{\columnwidth}
  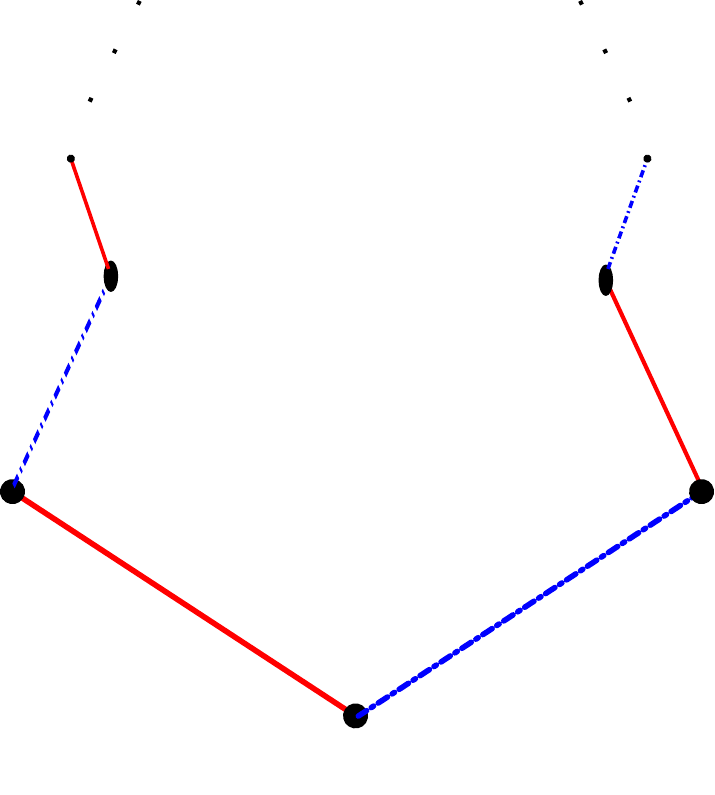
\end{subfigure}%
\begin{subfigure}{.4\textwidth}
  \centering
  \def\svgwidth{\columnwidth}
  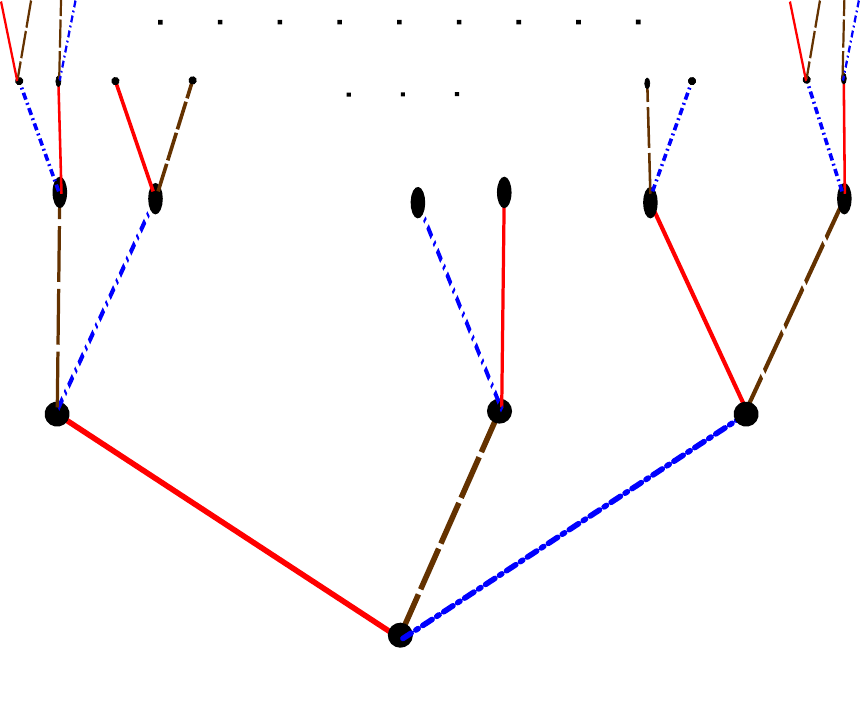
\end{subfigure}
\caption{Examples of an element in $\A$ (left) and an element in $\C$ for the case $n=4$}
\label{fig:elementsofCD}
\end{figure}


A regularly labeled $k$-subtree of $\Delta$ is determined by its color set $I \in \nck$ and any one of its vertices, $v$, hence  we can denote it by $$[v;I] \in T^I(\Delta).$$ The vertex $v$ is not unique, nonetheless, there is a unique vertex $v^* \in V([v;I])$ which is closest to the root $v_0 \in V(\Delta)$ so $d([v;I])=d(v^*)=l(\alpha)$ where $l(\alpha)$ is the length of the geodesic $\alpha$ from $v_0$ to $v^*$. Hence, $[v;I]$ inherits a rooted structure, with root $v^*$.

For convenience, we adopt the following conventions. Elements of $\A$ will be denoted with Greek letters $\alpha, \beta, \gamma, \delta$. Elements of $\C$ will be denoted by upper case letters $X, Y$ or $Z$.

Suppose that $I \in \nck$. Then
\[ [v;I]=\bigcap_{j \in \hatI} [v; \{\hat{j}\}].
\]
In particular,
\[
v=[v;\emptyset]=\bigcap_{j =1}^n [v; \{\hat{j}\}]=\bigcap_{j =1}^n X_j,
\]
where $X_j:=[v; \{\hat{j}\}] \in \C_j$.

Let $e \in E_{\{i\}}(\Delta)$ and $V(\bar{e})=\{v, v'\}$, that is, $e$ is labeled by $i$ and $v$ and $v'$ are the vertices incident to $e$. For $ j \in \setn$, let $X_j:=[v; \{\hat{j}\}]$ and $X_j':=[v'; \{\hat{j}\}]$. Then  $X_j=X_j'$ for $j \neq i$, but $X_i \neq X_i'$.
In this case, we write
 $$ e \leftrightarrow (\{X_1,\cdots,\hat{X}_i,\cdots,X_n\};\{X_i,X'_i\}). $$
 We shall also write
 $$ \vec{e} \leftrightarrow (\{X_1,\cdots,\hat{X}_i,\cdots,X_n\};X_i \rightarrow X'_i) $$
 for the directed edge $\vec{e}$ on $e$ pointing away from $X_i$ and towards $X'_i$.



 \subsection{Hurwitz maps} We define functions from $\C=\Tnminusone$ to $\CC$ induced from the vector-valued Hurwitz functions $\Phi$ defined in \S \ref{ss:cayley}.

\begin{defn} A {\it Hurwitz map} is a function $\phi: \C \rightarrow \mathbb C$ satisfying the following edge relations:
 for each $e \in E_{\{i\}}(\Delta)$, if $e \leftrightarrow (\{X_1,\cdots,\hat{X}_i,\cdots,X_n\};\{X_i,X'_i\})$ then
 \begin{equation}\label{eqn:edge-phi}
 \phi(X_i) + \phi(X'_i) = \prod_{j\neq i} \phi(X_j).
 \end{equation}
\end{defn}
 \noindent {\bf Convention}: For a fixed Hurwitz map $\phi$ and for each $X\in \C$, use the corresponding lower case letter to denote $\phi(X)$, for example, $x_i=\phi(X_i)$, $y=\phi(Y)$. We will adopt this convention in the rest of this paper.

Equation (\ref{eqn:edge-phi}) above can then be written as
 \begin{equation}\label{eqn:edge}
 x_i+x'_i=\prod_{j\neq i} x_j.
 \end{equation}

 It follows from the edge relations (\ref{eqn:edge-phi}) that a Hurwitz map $\phi$ also satisfies the vertex relations:
 there exists $\mu=\mu(\phi) \in \mathbb C$ such that for every vertex $v=\bigcap_{i=1}^{n}X_i \in V(\Delta)$, where $X_i=[v; \{\hat{i}\}]$,
 \begin{equation}\label{eqn:vertex}
 x_1^2 + \cdots + x_n^2 - x_1\cdots x_n = \mu.
 \end{equation}
 In this case, we also call $\phi$ a $\mu$-Hurwitz map.

The correspondence between the vector-valued Hurwitz maps $\Phi:V(\Delta) \rightarrow \CC^n$ and the Hurwitz maps $\phi: \C \rightarrow \mathbb C$ is as follows: Given a vector-valued Hurwitz map $\Phi: V(\Delta) \rightarrow \CC^n$, we get a Hurwitz map $\phi$ where $\phi(X_i)=\phi([v; \{\hat{i}\}])$ is the $i$-th coordinate of $\Phi(v)$. It is easy to check that this is well-defined, and that the edge relations (\ref{eqn:edgerelation}) implies the edge relations (\ref{eqn:edge}). Conversely, given a Hurwitz map $\phi$ we can get a vector-valued Hurwitz map $\Phi:V(\Delta) \rightarrow \CC^n$ by defining
\[ \Phi(v)=(\phi([v;\{\hat{1}\}]), \cdots, \phi([v;\{\hat{n}\}])) \in \CC^n
\]
and the edge relations (\ref{eqn:edge}) imply the edge relations (\ref{eqn:edgerelation}).

 We denote by $\Hur$ (resp. $\Hur_{\mu}$) the set of all Hurwitz (resp. $\mu$-Hurwitz) maps. Since the vector-valued Hurwitz maps are parametrized by $\CC^n$, $\Hur$ identifies with $\CC^n$ and  $\Hur_{\mu}$ identifies with the variety defined by (\ref{eqn:vertex}).

 \subsection{Dihedral Hurwitz maps}

 A Hurwitz map $\phi$ is called {\it dihedral} if at some (and hence every) vertex $v=X_1\cap \cdots\cap X_n$ where $X_i=[v;\{\hat{i}\}]$,
 $$ x_i=x_j=0 $$
 for some distinct indices $i$ and $j$. In this case we have $\sum_{k\neq i,j}x_k^2=\mu$.

 \subsection{Extended Hurwitz maps}\label{ss:extendedhurwitz}
 We can extend  a Hurwitz map $\phi: \C \rightarrow \mathbb C$ to  maps
 $T^{|k|} (\Delta)\rightarrow \mathbb C$, where  $0 \le k \le n-2$, which we still denote by $\phi$, defined by
 \begin{equation}\label{eqn:phi^k}
 \phi(X_{i_1}\cap \cdots \cap X_{i_{n-k}}) = x_{i_1} \cdots x_{i_{n-k}}.
 \end{equation}
 In particular, we have the extended map $\phi: \A(=\Tabstwo) \rightarrow \mathbb C$ where if $\gamma=[v;\{i,j\}] \in \A_{ij} (= T^{\{i,j\}}(\Delta))$ and $v=X_1 \cap \cdots \cap X_n$, $X_k =[v;\{\hat{k}\}]\in \C_k$, then
\begin{equation}
\phi(\gamma)=\phi([v;\{i,j\}])=\prod_{k \neq i,j} x_k.
\end{equation}
This is precisely the (extended) Hurwitz map $\phi$ from $\A$ to $\CC$ defined in the introduction induced from $\Phi$. Note that $\phi$ takes the value zero for some $X \in \C$ if and only if it takes the value zero for some $\gamma \in \A$.

For convenience, for the rest of this section and the next, we will think of a Hurwitz map $\phi$ as a map from $\C \cup \A \rightarrow \CC$ given by the definitions above. Furthermore, we will mostly be considering Hurwitz maps which take non-zero values, in particular, the dihedral map will usually be excluded.

\subsection{Directed trees}\label{ss:directedtrees} We next show how a Hurwitz map induces a direction for each edge $e \in E(\Delta)$ to give a directed tree.

A {\em directed tree} $\vec{\Delta}$ is a tree $\Delta$ where every edge $e \in E(\Delta)$ is assigned a direction. We mostly consider the case where $\Delta$ is $n$-regular, with $n \ge 3$.




\subsection{Sinks, merges and forks}\label{ss:sinkfm} The vertices of a directed tree can be classified as follows:

 For a directed tree $\vec \Delta$,  a vertex $v$ such that all incident edges are directed towards $v$ is called a {\em sink}. A  vertex $v'$ such that all incident edges but one are directed towards $v'$ is called a {\em merge}. See Figure \ref{fig:4sinkmerge} for examples of a sink and a merge. A vertex $v''$ such that more than one of the incident edges is directed away from $v''$ is called a {\em generalized fork}, or for simplicity, just {\em fork}. See Figure \ref{fig:4forks} for examples of forks.

\begin{figure}[ht]
\centering
\begin{subfigure}{.4\textwidth}
  \centering
  \includegraphics[width=.65\linewidth]{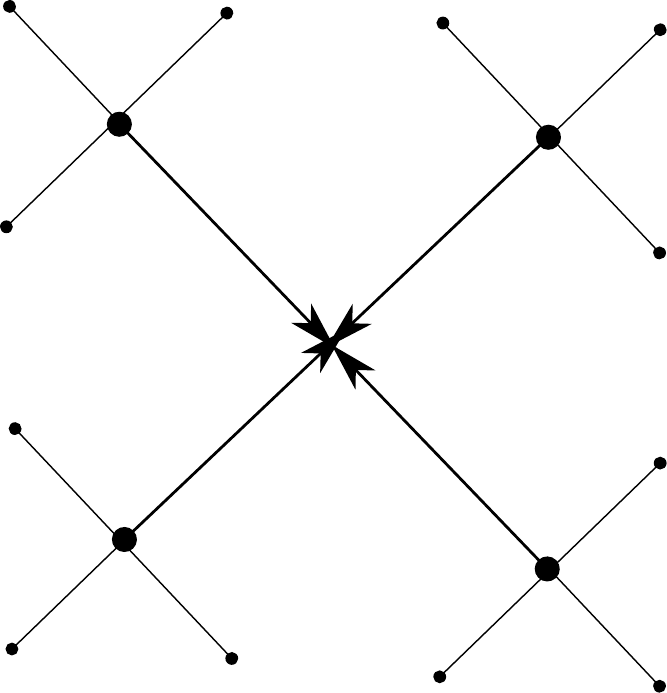}
\end{subfigure}%
\begin{subfigure}{.4\textwidth}
  \centering
  \includegraphics[width=.65\linewidth]{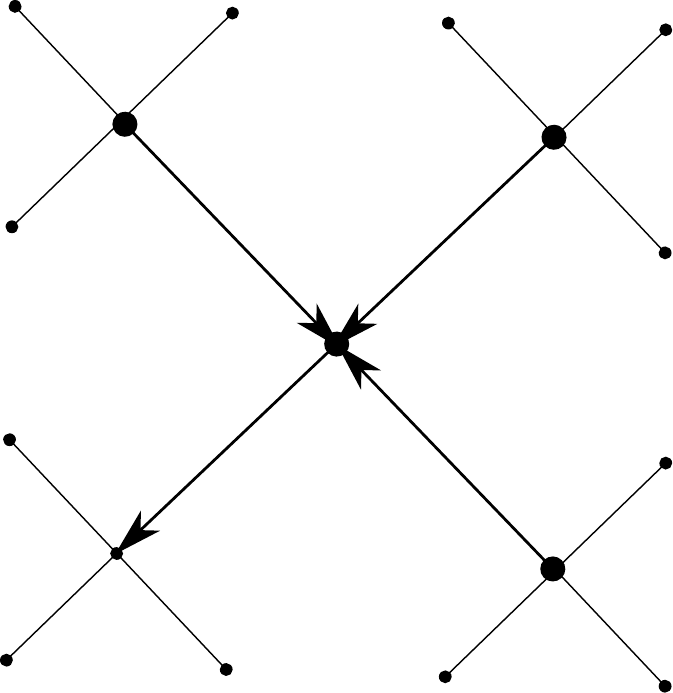}
\end{subfigure}
\caption{Examples of a sink in $V(\Delta)$ (left) and a merge for the case $n=4$}
\label{fig:4sinkmerge}
\end{figure}

\begin{figure}[ht]
\centering
\begin{subfigure}{.33\textwidth}
  \centering
  \includegraphics[width=.65\linewidth]{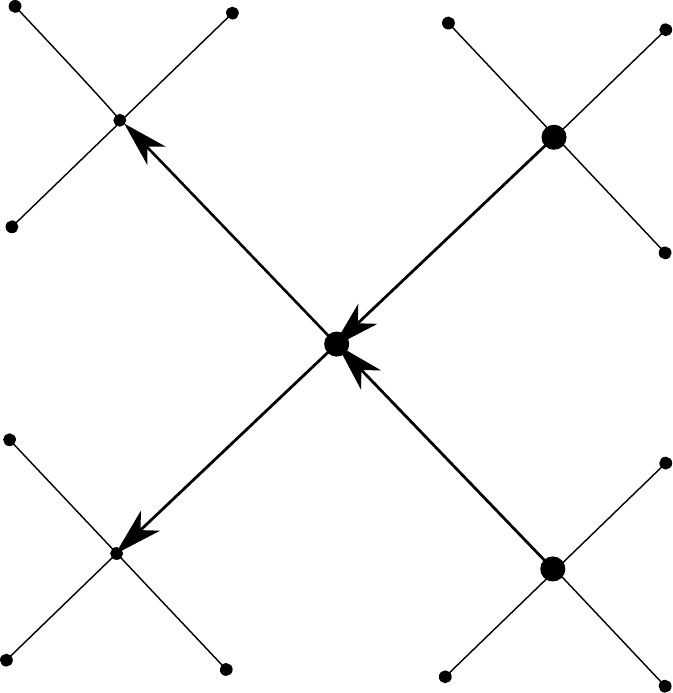}
\end{subfigure}%
\begin{subfigure}{.33\textwidth}
  \centering
  \includegraphics[width=.65\linewidth]{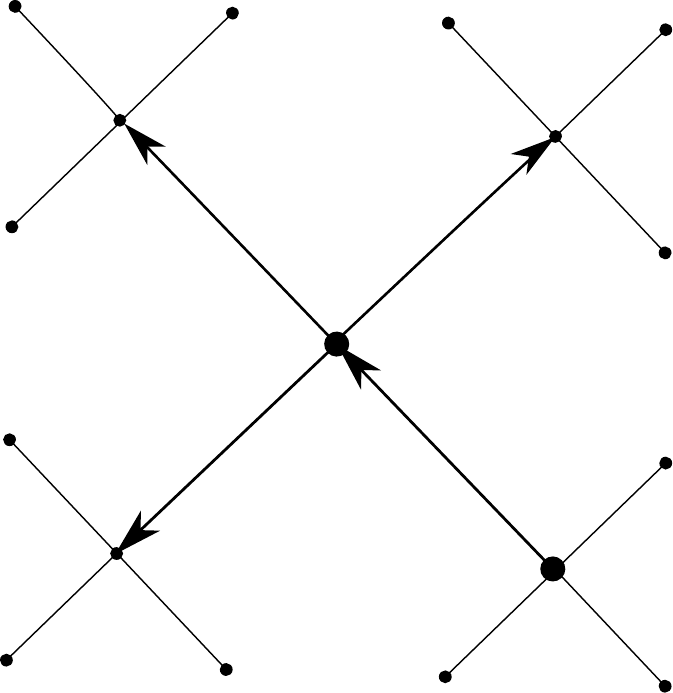}
\end{subfigure}%
\begin{subfigure}{.33\textwidth}
  \centering
  \includegraphics[width=.65\linewidth]{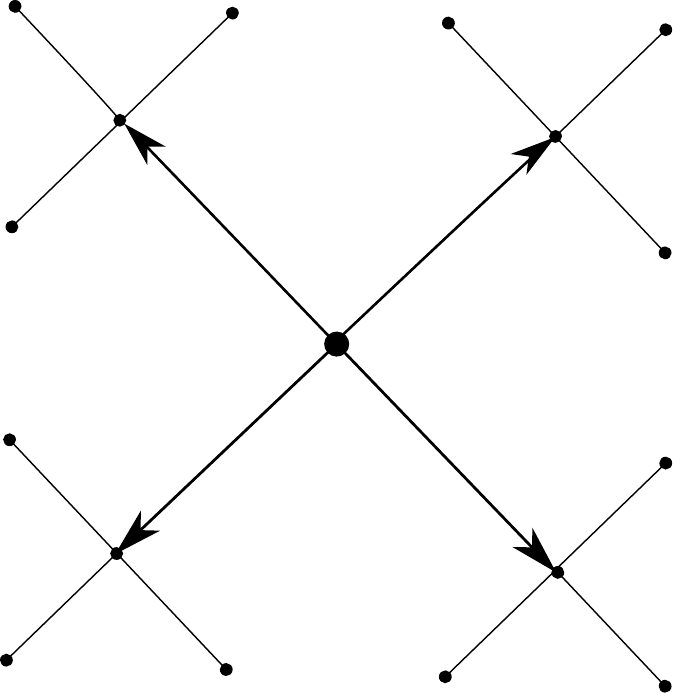}
\end{subfigure}
\caption{Examples of forks for the case $n=4$}
\label{fig:4forks}
\end{figure}
 \subsection{Induced directed tree $\vecDelta$}

 Given a Hurwitz map $\phi$ on $\C$, direct  each edge $e \in E(\Delta)$ where $e \leftrightarrow (\{X_1,\cdots,\hat{X}_i,\cdots,X_n\};\{X_i,X'_i\})$ as follows:

If $|\phi(X_i)|\neq |\phi(X'_i)|$, direct $e$ from $X_i$ to $X'_i$ if $|\phi(X_i)|>|\phi(X'_i)|$;  from $X'_i$ to $X_i$ if $|\phi(X'_i)|>|\phi(X_i)|$. In this case, we say the edge is {\em decisive}. If $|\phi(X_i)|=|\phi(X'_i)|$, direct the edge $e$ arbitrarily, and say $e$ is {\em indecisive}. The resulting directed tree is denoted by $\vec{\Delta}_{\phi}$.

We note that for $\phi$ dihedral, every edge is indecisive, on the other hand, we will show later that for $\phi$ satisfying the Bowditch conditions (\S \ref{ss:Bowditchconditions}), almost every edge is decisive and directed towards a finite subtree $T$ of $\Delta$.

 \subsection{Edge relations for alternating geodesics}\label{ss:relgeodesics}

 Let $e \in E_{\{i\}}(\Delta)$ with incident vertices $v$ and $v'$, and let $j,k \in \setn$ be different from $i$. Let
 $\alpha = [v;\{j,k\}]$, $\beta = [v;\{i,j\}]= [v';\{i,j\}]$, $\gamma = [v;\{i,k\}]= [v';\{i,k\}]$ and $\alpha' = [v';\{j,k\}]$.
 It follows from the edge relation (\ref{eqn:edge-phi}) that the extended Hurwitz map $\phi$ on $\A$ satisfies (Figure \ref{fig:edgerelation2})
 \begin{equation}\label{eqn:path-relation}
 \phi(\alpha) + \phi(\alpha') = \phi(\beta)\phi(\gamma).
 \end{equation}

 \begin{figure}[ht]
 \def\svgwidth{.75\columnwidth}
  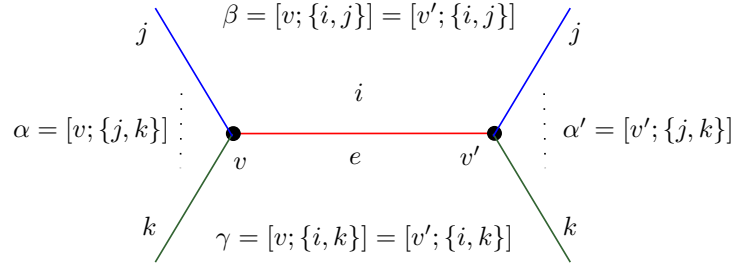
  \caption{The edge relations for alternating geodesics}
  \label{fig:edgerelation2}
 \end{figure}

 \subsection{Edge-connectedness of $\A_{\phi}(K)$}

 We say that a subset $\mathcal P$ of $\A$ is {\it edge-connected} if for every pair of elements $\alpha, \beta \in \mathcal P$,
 $\alpha$ is edge-connected to $\beta$ within $\mathcal P$, that is, there exists a sequence $\gamma_1, \cdots, \gamma_m \in \mathcal P$ such that
 $\gamma_1=\alpha$, $\gamma_m=\beta$,  and each adjacent pair $\gamma_i$ and $\gamma_{i+1}$ share a common edge.

 For a Hurwitz map $\phi$ and $K\ge 0$, we define  $\A_{\phi}(K) \subset \A$ by
 $$ \A_{\phi}(K) = \{ \gamma \in \A \mid |\phi(\gamma)| \le K \}.$$

The following two results  are fundamental to our analysis:


 \begin{lem}\label{lem:2away}
Let $\phi$ be a non-dihedral Hurwitz map from $\C \cup \A$ to $\CC$ and suppose that for $v \in V(\Delta)$, $v=\bigcap_{i=1}^n X_i$ where $X_i=[v;\{\hat{i}\}]$, and $e_i \in E_{\{i\}}(\Delta)$ are the edges incident to $v$, $i=1, \cdots, n$.
\begin{enumerate}
\item [(i)] (Fork Lemma) If the $\phi$-induced arrows on  $e_i$ and $e_j$
 both point away from $v$, then $\alpha :=[v;\{i,j\}] \in \A_{\phi}(2)$;
\item[(ii)] Let $\alpha=[v;\{i,j\}]$, $\beta=[v;\{k,l\}]$, $\gamma=[v;\{i,k\}]$ and $\delta=[v;\{j,l\}]$ where $i,j,k,l \in \setn$ are all distinct. If $\alpha, \beta \in  \A_{\phi}(K)$, then either $\gamma \in \A_{\phi}(K)$ or $\delta \in \A_{\phi}(K)$. In particular, $\alpha$ and $\beta$ are edge-connected within $\A_{\phi}(K)$;
\item[(iii)] Suppose that $\alpha=[v;\{i,j\}] \in \A_{\phi}(K)$, and the $\phi$-induced arrow on $e_k$ (where $k \neq i,j$) points away from $v$ towards the  head $v' \in V(\Delta)$. Then $\beta=[v';\{i,j\}] \in \A_{\phi}(K)$.
\end{enumerate}
 \end{lem}


\begin{proof}

 (i) Without loss of generality, assume $e_1$ and $e_2$ are directed away from $v$ by $\phi$ with heads $v'$ and $v''$ respectively. Together with the edge relation (\ref{eqn:edge}), we have $$ 2|x_1| \ge |x_1+x_1'|= |x_2 x_3 \cdots x_n|,$$  $$ 2|x_2| \ge |x_2+x_2'|=|x_1 x_3 \cdots x_n|,$$ where $x_1'$ and $x_2'$ are the $\phi$ values of $[v';\{\hat{1}\}]$ and $[v'';\{\hat{2}\}]$.  Therefore
\[
(|x_1|+|x_2|)(|x_3 \cdots x_n|-2) \le 0.
\]
Since  $\phi$ is not dihedral, $|x_1|+|x_2| >0$,  so $|x_3 \cdots x_n| \le 2$, thus $[v;\{1,2\}] \in \A_{\phi}(2)$.

 (ii) Since $\alpha,\beta \in \A_{\phi}(K)$, we have
$$|\prod_{t \neq i,j} x_t| \le K, \qquad |\prod_{t \neq k,l} x_t| \le K.$$ It follows that either
$$|\prod_{t \neq i,k} x_t| \le K,\quad  \hbox{or} \quad |\prod_{t \neq j,l} x_t| \le K.$$ Hence, either $\gamma=[v; \{i,k\}] \in \A_{\phi}(K)$ or $\delta=[v; \{j,l\}] \in \A_{\phi}(K)$.

 (iii) Writing $x_k'=\phi([v';\{\hat{k}\}])$, we have $$|\phi(\beta)|=|x_k'\prod_{t \neq i,j,k}x_t| \le |\prod_{t \neq i,j}x_t|=|\phi(\alpha)| \le K,$$
hence $\beta \in \A_{\phi}(K)$.
 \end{proof}

 \begin{prop}\label{prop:e-c}
 For any non-dihedral  Hurwitz map $\phi$  and any $K \ge 2$, the subset $\A_{\phi}(K)$ is edge-connected.
 \end{prop}

 \begin{proof}
 We prove by contradiction.
 Suppose that $\A_{\phi}(K)$ is not edge-connected. Then
 there exist $\alpha, \beta \in \A_{\phi}(K)$ such that $\alpha$ cannot be edge-connected to $\beta$ within $\A_{\phi}(K)$ and the distance $m$ between $\alpha$ and $\beta$ is minimal within all such pairs.
 Let the unique shortest path from $\alpha$ to $\beta$ be denoted by the vertex-edge sequence $\{ v_0$, $e_1$, $v_1$, $e_2$, $v_2, \cdots, e_m$, $v_m\}$
 where $v_0 \in V(\alpha)$ and $v_m \in V(\beta)$. We  consider the cases $m=0$, $m=1$ and $m>1$ separately.

 Case $m=0$: Then $\alpha$ and $\beta$ share a common vertex $v_0$.
However, by Lemma \ref{lem:2away} (ii), $\alpha$ and $\beta$ are edge connected within $\A_{\phi}(K)$ via $\gamma$ or $\delta$, a contradiction.

 Case $m=1$: Without loss of generality, we may assume that $e_1 \in E_{\{1\}}(\Delta)$ so $e_1 \leftrightarrow (\{X_2, \cdots,X_n\};\{X_1,X_1'\})$ and that $e_1$ is directed away from $v_0$, hence $|x_1| \ge |x_1'|$. By \ref{lem:2away} (iii), $\alpha':=[v_1;\{i,j\}] \in \A_{\phi}(K)$. Now, since $\alpha'$ and $\beta$ share a common vertex,  by Lemma \ref{lem:2away} (ii), $\alpha'$ and $\beta$ are edge connected within $\A_{\phi}(K)$. Let  $\gamma =[v_0;\{1,i\}]$ and $\delta =[v_0;\{1,j\}]$. By the  edge relation (\ref{eqn:path-relation}), we have $\phi(\alpha) + \phi(\alpha') = \phi(\gamma) \phi(\delta)$. It follows that either $|\phi(\gamma)| \le K$ or $|\phi(\delta)| \le K$, (we use here the fact that $K \ge 2$). Thus $\alpha$ is edge-connected to $\alpha'$ within $\A_{\phi}(K)$, and hence edge-connected to $\beta$ within $\A_{\phi}(K)$, a contradiction.


 Case $m>1$: We claim that $e_1$ is directed towards $v_0$. Otherwise, writing $\alpha=[v_0;\{i,j\}]$, by Lemma \ref{lem:2away} (iii), we have $\alpha'=[v_1;\{i,j\}] \in \A_{\phi}(K)$. Then, since $\alpha$ and $\beta$ are not edge connected within $\A_{\phi}(K)$,  $\alpha'$ is not edge connected to either $\alpha$ or $\beta$ within $\A_{\phi}(K)$. However, its distance from both $\alpha$ and $\beta$ is strictly less than $m$,  contradicting the minimality of $m$. Similarly, $e_m$ must be directed towards $v_{m}$. It follows that $v_k$ is a fork for some $1 \le k \le m-1$, hence by Lemma \ref{lem:2away} (i), $\gamma=[v_k;\{s,t\}] \in \A_{\phi}(2) \subset \A_{\phi}(K)$ where $e_{k-1} \in E_{\{s\}}(\Delta)$ and $e_{k} \in E_{\{t\}}(\Delta)$. Repeating the previous argument about minimality of $m$ produces the required contradiction. This completes the proof.
\end{proof}

 \subsection{Circular sets}\label{ss:circular} Given a finite subtree $T$ of $\Delta$, we define the {\em circular set of $T$}, $C(T) \subset \vec{E}(\Delta)$ as follows:  $\vec{e}\in C(T)$ if and only if
 the underlying edge $e$ meets $T$ in a single point, that point being the head of $\vec{e}$.  A simple example of a circular set is $C(\{v\})$, the circular set of a vertex $v$,  which consists of the $n$ directed edges all of whose heads are $v$. Also, if $\Delta_n \subset \Delta$ is the subtree spanned by the vertices $v \in V(\Delta)$ such that $d(v) \le n$, then $C(\Delta_n)$ is the set of edges $\vec e \in \vec E(\Delta)$ such that $d(\vec e)=n$ and $d(\mathrm{Head}(\vec e))=n$.

 \subsection{Subsets $\A^{0}(e)$, $\A^{-}(\vec{e})$ and $\A^{0-}(\vec{e})$ of $\A$}

 Given an edge $e \in E(\Delta)$, we define
$$\A^{0}(e)=\{ \gamma \in \A ~|~ e \in E(\gamma)\},$$
that is, the set consisting of those alternating geodesics containing $e$. Note that $|\A^0(e)|=n-1$.
 Given a directed edge $\vec{e} \in \vec{E}(\Delta)$, $\Delta \setminus \{e\}$ consists of two connected subtrees of $\Delta$. Let $\Delta^{-}(\vec{e})$ (respectively $\Delta^{+}(\vec{e})$)  be the   subtree of $\Delta$ that contains the tail (respectively the head) of $\vec{e}$.
 We define
\[\A^{-}(\vec{e}):=\{\gamma \in\A ~|~ \gamma \subset \Delta^{-}(\vec{e})\}, \quad \A^{+}(\vec{e}):=\{\gamma \in\A ~|~ \gamma \subset \Delta^{+}(\vec{e})\}.
\]
 We see that $\A=\A^{-}(\vec{e}) \sqcup \A^{0}(e)\sqcup \A^{+}(\vec{e})$.
 Furthermore, we write $\A^{0-}(\vec{e})$ for $\A^{0}(e)\sqcup \A^{-}(\vec{e})$.

 More generally, suppose $C=C(T) \subset \vec{E}(\Delta)$ is a circular set of directed edges where $T$ is a finite subtree of $\Delta$.
 Let $\A^{0}(C)=\bigcup_{\vec{e}\in C}\A^{0}(e)$, that is, the set of alternating geodesis meeting $C$ in at least one edge.
 Then $\A$ can be written as the disjoint union of $\A^{0}(C)$ and $\A^{-}(\vec{e})$ as $\vec{e}$ varies in $C$.

 \subsection{The Fibonacci function $F$ on $\A$}\label{ss:Fibonacci}

 We define the Fibonacci function $$F:=F_{v_0}: \A \rightarrow \mathbb N$$ relative to the root $v_0 \in V(\Delta)$ recursively by induction on $d(\gamma)$.
 First, define $F(\gamma)=1$ for every $\gamma \in \A$ where $d(\gamma)=0$, that is, if $v_0 \in V(\gamma)$.
 Next, if $d(\gamma)=s>0$, let $\gamma=[v^*;\{i,j\}]$ where $v^*\in V(\gamma)$ is the unique vertex closest to $v_0$ and let $e_k\in E_{\{k\}}(\Delta)$ be the unique edge incident to $v^*$ which lies on the geodesic  from $v^*$ to $v_0$, see Figure \ref{fig:fibonaccidefi}. Then $k \neq i,j$. Let $\alpha=[v^*; \{i,k\}]$ and $\beta=[v^*; \{j,k\}]$. By construction, $d(\alpha)<s$ and $d(\beta)<s$. We define
 $$ F(\gamma)=F(\alpha)+F(\beta). $$

 \begin{figure}[ht]
\centering
\def\svgwidth{.6\columnwidth}
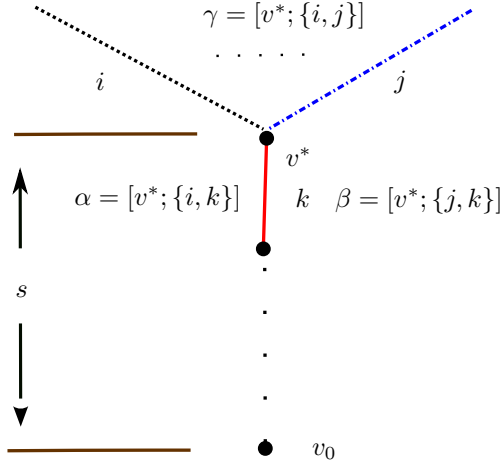
\caption{Alternating geodesics $\gamma, \alpha, \beta$ containing $v^*$ with $d(\alpha), d(\beta)<d(\gamma)$.}
\label{fig:fibonaccidefi}
\end{figure}

Since $v^*$ and $e_k$ are uniquely determined when $s>0$, $F$ is well defined for all $\gamma \in \A$. (Note that in \cite{bowditch1998plms, tan-wong-zhang2008advm}, the Fibonacci function was defined relative to an edge as opposed to  a vertex, there is no essential difference but since we are dealing with rooted trees here, we choose to define $F$ relative to  a vertex.)
We may define the Fibonacci function $F_v$ relative to any other vertex $v \in V(\Delta)$. It is easy to prove (see for example \cite{tan-wong-zhang2008advm}) that there exists a constant $k \ge 1$ depending only on $v$ and $n$ such that
 \begin{equation}\label{eqn:Fibonacciv}
 \frac{1}{k} F(\gamma) \le F_{v}(\gamma) \le k F(\gamma).
 \end{equation}
As we will only be interested in the growth rates of functions up to a multiplicative constant, we see that it does not matter which vertex is used to define the Fibonacci function.


\smallskip

 \subsection{ Sierpinski simplices}\label{ss:sierpinski} We now give a geometric interpretation for the Fibonacci function which is of independent interest. It will  also allow us to bound the multiplicity function $\mu:\NN \rightarrow \NN$ which counts the number of times a number $m$ occurs as $F(\gamma)$, $\gamma \in \A$. Let $\{\vec e_1, \cdots , \vec e_n\}=C(\{v_0\})$ be the circular set of $\{v_0\}$. By symmetry, the Fibonacci function looks the same for each of $\A^{0-}(\vec e_i)$, the set of alternating geodesics either passing through $e_i$ or contained in the subtree obtained by removing $e_i$ and containing the tail of $\vec e_i$. So, without loss of generality we may just look at $\A^{0-}(\vec e_n)$. We will define a function $$S: \A^{0-}(\vec e_n) \rightarrow {\Z}^{n-1}$$ as follows:
For $\gamma_i=[v_0; \{i,n\}] \in \A^0(\vec e_n)$, $i=1, \cdots, n-1$, define $S(\gamma_i)=\mathbf{e}_i=(0,\cdots, 1,\cdots, 0)\in{\mathbb Z}^{n-1}$, the standard basis vector with 1 in the $i$th entry and 0 elsewhere. Now for $\gamma \in \A^{-}(\vec e_n)$, define
\begin{equation}\label{eqn:Sierpinski}
S(\gamma)=S(\alpha)+S(\beta)
\end{equation}
where $\alpha, \beta$ are related to $\gamma$ as in the definition of $F$.
Clearly, by construction, $F(\gamma)$ equals the sum of the entries of $S(\gamma)$ for $\gamma \in \A^{0-}(\vec e_n)$. Composing with the projective map $p: \RR^{n-1}\backslash\{\mathbf{0}\} \rightarrow \RPnminustwo$, the image of $\A^{0-}(\vec e_n)$ under $p \circ S$ in $\RPnminustwo$  consists  precisely of the vertices of a projective Sierpinski $(n-2)$-simplex where in the construction of the simplex, each edge is subdivided according to the rule given by (\ref{eqn:Sierpinski}). See Figure \ref{fig:4Sierpinski} for the case $n=4$ and the first few iterations. In particular,  $p \circ S(\gamma)$ is in the convex hull of $p \circ S(\alpha)$ and $p \circ S(\beta)$. Thus, $p \circ S$ and hence  $S$ are both injective.

Using the standard topology on $\RPnminustwo$ and the injection of $\A$ into $\RPnminustwo$ by $p \circ S$, we can consider the closure $\overline{p \circ S(\A)} \subset \RPnminustwo$ which is the projective Sierpinski $(n-2)$-simplex.

To obtain all of $\A$, we take $n$-copies of the projective Sierpinski $(n-2)$-simplices constructed, and identify the boundary vertices in pairs, coming from each $[v_0, \{i, j\}]$ which occurs as the boundary vertex of a pair of simplices.
This set can be identified with $\mathrm{End}(\Delta)/\sim$, where  $\mathrm{End}(\Delta)$ is the set of ends of $\Delta$. The ends of $\Delta$ can be represented by infinite geodesic rays emanating from $v_0$. Two elements of $\mathrm{End}(\Delta)$ represented by two such infinite geodesic rays $\alpha$ and $\beta$ are equivalent if both eventually converge to the same $\gamma \in \A$, that is, the symmetric difference $\alpha \triangle \beta=\gamma$. The rational ends correspond to the elements of $\A$ and the irrational ends correspond to the those paths which do not limit to any $\gamma \in \A$.

For a given Hurwitz map $\phi$, it would be interesting to study  the set of end invariants as defined in \cite{tan-wong-zhang2008ajm}, this says something about the dynamics of the action of $\Gamma_n$ on $\CC^n$. We plan to do this in a future project. Another interesting direction for further investigations is to look for similar underlying geometric structures for the sets of other regular subtrees $T^{|k|}(\Delta)$ where $1 \le k\le n-1$.

\begin{figure}[ht]
\centering
\includegraphics[width=0.9\linewidth]{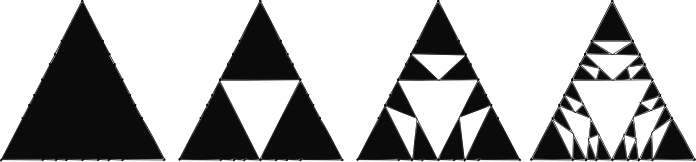}
\caption{The Sierpinski triangle}
\label{fig:4Sierpinski}
\end{figure}

 The multiplicity function $\mu:\NN \rightarrow \NN$ where $\mu(m)= \#\{\gamma \in \A~|~F(\gamma)=m\}$  seems to be an interesting function from a number theoretic and combinatorial point of view. When $n=3$, $\mu(m)=3\varphi(m)$ where $\varphi$ is the Euler's totient function. It would be interesting to obtain a closed form for this function for $n>3$. However, for our purposes, it suffices to find an upper bound for $\mu(m)$. We first note that $\mu(1)=\frac{1}{2}n(n-1)$ since this is the number of alternating geodesics passing through $v_0$. For $m \ge 2$, we have the following facts:
\begin{enumerate}
\item[(i)] the function $F$ is symmetric on the branches $\A^{0-}(\vec e_i)$;
\item[(ii)]  $S$ is injective;
\item[(iii)]  $F(\gamma)$ is given by the sum of the entries of $S(\gamma)$;
\item[(iv)] the entries of $S(\gamma)$ are non-negative integers less than $m$ (so $m$ choices) since at least two entries are non-zero;
\item[(v)] we only need to know the first $n-2$ entries of $S(\gamma)$ since the sum of the entries is $m$ (so $m^{n-2}$ choices).
\end{enumerate}
From the above, we easily deduce the following, which allows us to get the absolute convergence of certain series:

 \begin{lem}\label{lem:countgamma}
For $m \ge 2$, the multiplicity function $\mu$ satisfies $\mu(m) \le n m^{n-2}$.

 \end{lem}

 \begin{prop}\label{prop:fibconvergence}
 The infinite sum $\sum_{\gamma \in \A}(F(\gamma))^{-s}$ is convergent for $s > n-1$.
 \end{prop}

 \begin{proof}
 By Lemma \ref{lem:countgamma}, we see that
 \begin{eqnarray*}
 \sum_{\gamma\in\A}(F(\gamma))^{-s}
 \le \frac{1}{2}n(n-1)+n  \sum_{m=2}^{\infty}m^{n-2-s} < +\infty
 \end{eqnarray*}
 since $s>n-1$.
 \end{proof}

 \subsection{Growth along an alternating geodesic}\label{ss:growth}

 Let $\gamma$ be an alternating $\{i,j\}$-geodesic with vertex set $\{v_m\}$  and edge set $\{e_m\}$ where $e_m$ is incident to $v_m$ and $v_{m+1}$. Further, for $m \in \Z$,  suppose that $e_{2m+1} \in E_{\{i\}}(\Delta)$,  $e_{2m} \in E_{\{j\}}(\Delta)$, see Figure \ref{fig:alternatingray}.

\begin{figure}[ht]
\centering
\def\svgwidth{\columnwidth}
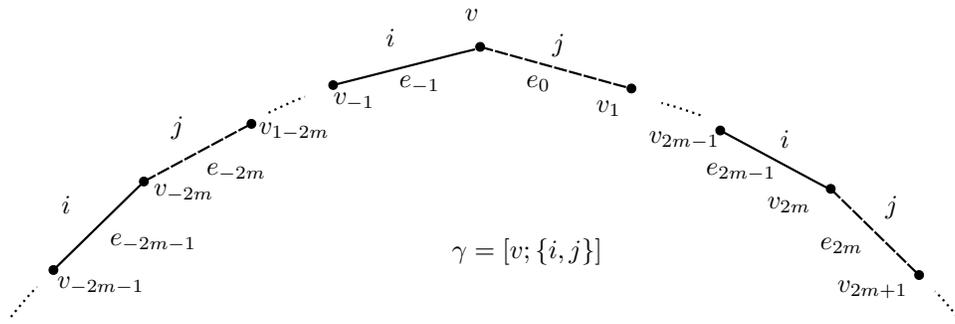
\caption{An alternating geodesic $\gamma$}
\label{fig:alternatingray}
\end{figure}

Let $Y_m \in \C$ such that
$$Y_m \cap \gamma=\bar{e}_m, \quad m\in \Z.$$
Specifically,
\[Y_{2m+1}=[v_{2m+1};\{\hat{j}\}]=[v_{2m+2};\{\hat{j}\}] \in \C_j,
\]
\[Y_{2m}=[v_{2m};\{\hat{i}\}]=[v_{2m+1};\{\hat{i}\}] \in \C_i, \quad m \in \Z.
\]

Recall the definition of the (extended) Hurwitz function $\phi$ and the square-sum weight $\sigma$ on $\gamma \in \A$ as follows: Let $v_0=\cap_{m=1}^n X_k$, $X_k \in \C_k$, $k=1, \cdots, n$:

 \[
 \phi(\gamma)=x:=\prod_{k \neq i,j} x_k :=\lambda+\lambda^{-1}
 \]
 where $|\lambda|\ge 1$ and
\[
\sigma(\gamma) =\sum_{k \neq i,j} x_k^2.
\]
 Note that $|\lambda|=1$ if and only if $x \in [-2,2]$, while $\lambda=\lambda^{-1}$ if and only if $x \in \{-2,2\}$.

 By the edge and vertex relations \eqref{eqn:edge} and \eqref{eqn:vertex}, we have
\[y_{m-1}+y_{m+1}=x y_{m}\] and
 \[
 y_m^2+y_{m+1}^2+\sigma(\gamma)=x y_m y_{m+1} + \mu
 \]
 for all $m\in\mathbb Z$.
 If $x\notin \{-2,2\}$, then solving the difference equation gives
$$ y_m = A \lambda^{m} + B\lambda^{-m}, \quad m\in\mathbb Z $$
where $A=(y_1-y_0\lambda^{-1})/(\lambda-\lambda^{-1})$ and $B=(y_0\lambda-y_1)/(\lambda-\lambda^{-1})$.
Furthermore,
 \[AB=\frac{\sigma(\gamma)-\mu}{x^2-4}.\]

For $x = \pm 2$ the situation is even simpler, in that case $y_m$ is linear in $m$. Consolidating everything, we have the following result which describes the growth of $y_m$ along $\gamma$ (compare \cite{tan-wong-zhang2008advm} Lemma 3.9).

 \begin{lem} \label{lem:geodesicalong}
With the notation introduced above, we have, for all $m \in \Z$,

 (1) If $\sigma(\gamma) = \mu$, then $y_m = y_0 \lambda^m$  or $y_m = y_0 \lambda^{-m}$.

 (2) If $\phi(\gamma)\not\in [-2,2]$ and $\sigma(\gamma) \neq \mu$  then $|y_m|$ grows exponentially as $|m| \rightarrow \infty$.

 (3) If $\phi(\gamma)\in (-2,2)$ then $|y_m|$ remains bounded.

 (4) If $\phi(\gamma)=2$, then $y_m=y_0+m(\mu-\sigma(\gamma))^{1/2}$.

 (5) If $\phi(\gamma)=-2$, then $y_m=(-1)^{m}(y_0+m(\mu-\sigma(\gamma))^{1/2})$.
 \end{lem}

 \subsection{Sink estimates}
In general, at a vertex $v=X_1 \cap \cdots \cap X_n \in V(\Delta)$ where $X_i=[v;\{\hat{i}\}]$,
 a $\mu$-Hurwitz map $\phi$ can take values $x_1,\cdots,x_n \in \mathbb C$ each of very large modulus. Furthermore,  for the alternating geodesics $\gamma$ passing through $v$, $|\phi(\gamma)|$  may also be very large.
 However, if $v$ is a sink of the $\phi$-directed tree $\vec{\Delta}_{\phi}$, then the minimum of $|x_1|,\cdots,|x_n|$ and $|\phi(\gamma)|$ where $v \subset \gamma$ have  upper bounds dependent only on $\mu$. We have:

 \begin{prop}\label{prop:sinkestimate}
Let $\phi \in \Hur_{\mu}$ where $\mu \in \CC$. Suppose that the $\phi$-directed tree $\vec \Delta_{\phi}$
 has a sink at $v=X_1 \cap \cdots \cap X_n$ where $X_i=[v;\{\hat{i}\}]$. Let $\gamma_{i,j}=[v;\{i,j\}]$, where $ 1 \le i <j \le n$. Then there exists constants $M(\mu), N(\mu)>0$ such that
$\min \{ |x_1|,\cdots,|x_n| \} \le N(\mu)$ and  $\min \{ |\phi(\gamma_{i,j})|, ~ 1 \le i <j \le n \} \le M(\mu)$.

 \end{prop}

 The proof of Proposition \ref{prop:sinkestimate} is essentially the same as that of Lemma 3.8 in \cite{tan-wong-zhang2008advm} and we omit it.

 \section{Bowditch conditions and domain of discontinuity}\label{s:domainofdis}

 \subsection{The Bowditch conditions}\label{ss:Bowditchconditions}
 We denote by $\B$ the set of all  (extended) Hurwitz maps $\phi:\C \cup \A \rightarrow\mathbb C$
 which satisfy the {\it Bowditch conditions} (B1) and (B2) below:
 \begin{itemize}
 \item[(B1)] $\{\gamma \in \A \mid \phi(\gamma) \in [-2,2] \} = \emptyset$;
 \item[(B2)] $\A_{\phi}(K)$ is finite (possibly empty) for some $K>2$.
 \end{itemize}
 We also denote by $\B_{\mu}\subset \B$ the subset consisting of $\mu$-Hurwitz maps. Recalling that $\Hur$ identifies with $\CC^n$, we  denote by $\D$ and $\D_{\mu}$ the corresponding sets of $\mathbfit{a}\in \CC^n$ which induces Hurwitz maps $\phi \in \B$ and $\B_{\mu}$ respectively.

 Note that if the set described in (B2) is non-empty, it is edge-connected by proposition \ref{prop:e-c}. Also, if $\A_{\phi}(K)$ is finite for some $K>2$, it is finite for $\A_{\phi}(K')$ for any $K'<K$. We shall see later that if (B1) and (B2) holds for some $K>2$, then $\log^+|\phi|$ has Fibonacci growth which implies that (B2) also holds for any $K>2$.

 \smallskip

 The rest of this section will be devoted to proving results about $\B$ which would imply Theorem \ref{thm:domainofdis}, that is, $\D$ is a non-empty,
 invariant open subset of $\CC^n$ on which  $\Gamma_n$ acts properly discontinuously. This answers Question 2 in the introduction. We first note that the definition of $\B$ implies that $\D$ and $\D_{\mu}$ are $\Gamma_n$-invariant since for any $g \in \Gamma_n$, $\phi_{g(\mathbfit{a})}$ satisfies the Bowditch conditions (B1) and (B2) if and only if $\phi_{\mathbfit{a}}$ satisfies (B1) and (B2).

 \subsection{Escaping ray}

 The following lemma guarantees that for any $\phi\in\B$, there is no escaping ray consisting of $\phi$-directed edges.

 \begin{lem}\label{lem:escpath}
 Suppose $\alpha $ is an infinite geodesic ray in $\Delta$ consisting of a sequence of edges
 such that each edge is $\phi$-directed towards the next.
 Then either $\alpha$ is eventually contained in some $\gamma\in\A$ with $\phi(\gamma)\in[-2,2]$,
 or $\alpha$ meets  infinitely many $\gamma\in\A_{\phi}(K)$ for any given $K >2$.
 \end{lem}

 \begin{proof}
 We first claim that for any $K>2$,  $\alpha$ meets some $\gamma \in \A_{\phi}(K)$.
 Suppose the sequence of edges of the ray $\alpha$ is $\{e_m\}_{ m\in\mathbb N}$ where for each $m \in \mathbb N$, $e_m$  is directed from the vertex $v_m$ to $v_{m+1}$. Since $\alpha$ is infinite, there exists $i \in [n]$ such that infinitely many $e_m$ are labeled by $i$; see Figure \ref{fig:escapingray}. The subsequence of $\{e_m\}_{m\in\mathbb N}$ consisting of $i$-edges connects a sequence  $\{Y^i_l\}_{l=1}^{\infty}$ of elements in $\C_i$, the $(n-1)$-regular subtrees labeled by $\{\hat{i}\}$. From the direction of the edges,  the sequence $\{|y^i_l|\}_{l=1}^{\infty}$ is monotonically decreasing, and hence has a limit. In particular, for any arbitrarily small $\varepsilon > 0$, there exists $m \in \NN$ such that $|y_{m+1}^i|\le |y_{m}^i|\le |y_{m+1}^i|+\varepsilon$.
 We focus our attention to the part of $\alpha$ consisting of the edge joining $Y_{m}^i$ to $Y_{m+1}^i$, which we call $e$, and the next edge which we call $f$ for simplicity. Also, to simplify notation, let  $e\in E_{\{i\}}(\Delta)$ be directed from $v$ to $v'$ and $f\in E_{\{j\}}(\Delta)$ ($j \neq i$) be directed from $v'$ to $v''$, see Figure \ref{fig:escapingrayedge}.

\begin{figure}[ht]
\begin{minipage}{.55\textwidth}
  \centering
  \def\svgwidth{\columnwidth}
  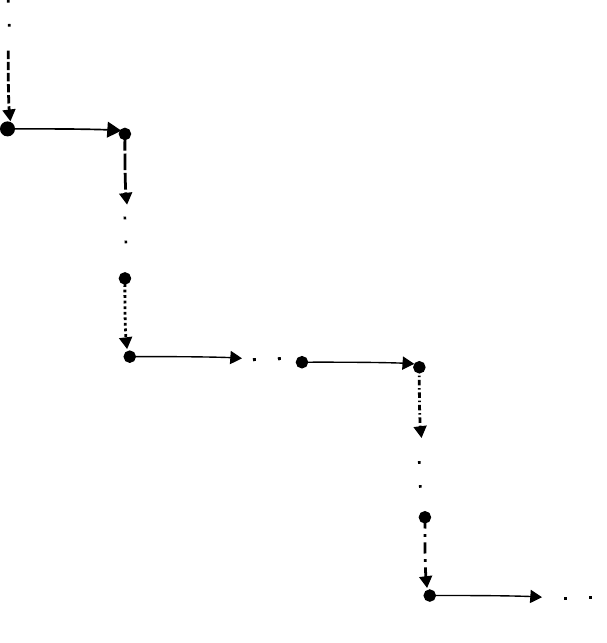
  \caption{Escaping ray}
  \label{fig:escapingray}
\end{minipage}%
\begin{minipage}{.4\textwidth}
  \centering
  \def\svgwidth{\columnwidth}
  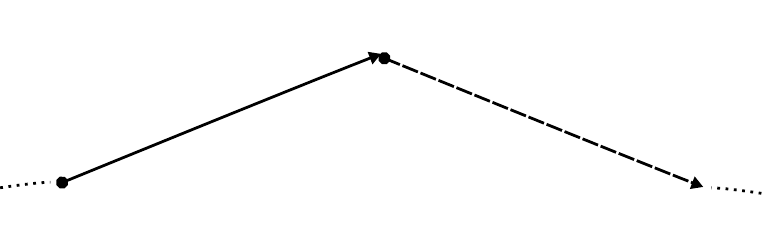
  \caption{Neighboring edges in an escaping ray}
  \label{fig:escapingrayedge}
\end{minipage}
\end{figure}

We have
$$v=\bigcap_{m \in [n]} X_m:=\bigcap_{m \in [n]} [v;\{\hat{m}\}], \quad
v'=\bigcap_{m \in [n]} X_m', \quad
v''=\bigcap_{m \in [n]} X_m'',$$
 and
\begin{equation}\label{eqn:epsiloninq}
|x_i'|\le |x_i| \le |x_i'|+\varepsilon.
\end{equation}
Note that $$x_j=x_j', \quad x_i'=x_i'', \quad \hbox{and}$$
$$ x_m=x_m'=x_m'', \quad \hbox{ for} \quad m \neq i,j. $$

For $m \neq i$, let $\gamma_m=[v';\{i,m\}]$. If $\gamma_m \in \A_{\phi}(2)$ for some $m$, we are done, so we may assume that $\gamma_m \not\in \A_{\phi}(2)$ for all $m \neq i$. Then  $$|\prod_{k \neq m,i}x_k'|>2 \quad \hbox{ for all} \quad m \neq i.$$ Taking the product of all these inequalities, we get
\begin{equation}\label{eqn:edgeinqone}
 |\prod_{m \neq i}x_m'|^{n-2} >2^{n-1} \Longrightarrow |\prod_{m \neq i}x_m'| >2^{(n-1)/(n-2)}.
\end{equation}
On the other hand, from the edge relation across $e$ and (\ref{eqn:epsiloninq}), we get
\begin{equation}\label{eqn:edgeinqtwo}
|\prod_{m \neq i}x_m'|=|x_i+x_i'| \le |x_i|+|x_i'|\le 2|x_i'|+\varepsilon
\end{equation}
Combining inequalities (\ref{eqn:edgeinqone}) and (\ref{eqn:edgeinqtwo}), we get
\begin{equation}\label{eqn:edgeinqthree}
|x_i'|>2^{1/(n-2)}-\varepsilon/2.
\end{equation}
Hence, if $\varepsilon$ is sufficiently small, $|x_i'|>1$. We now claim that $\gamma_j=[v';\{i,j\}] \in \A_{\phi}(2+\varepsilon)$. From the edge relation across $f$, and the direction of $f$, we get:
\begin{equation}\label{eqn:edgeinqfour}
|\prod_{m \neq j} x_m'| =|x_j'+x_j''| \le |x_j'|+|x_j''| \le 2|x_j'|.
\end{equation}
Multiplying (\ref{eqn:edgeinqtwo}) and (\ref{eqn:edgeinqfour}) gives
\begin{eqnarray*}\label{eqn:edgeinqfive}
|x_i' x_j'(\prod_{m \neq i,j} x_m')^2| &\le & 4|x_i'x_j'|+2|x_j'|\varepsilon \\
\Longrightarrow |\prod_{m \neq i,j} x_m'|^2 & \le& 4+\frac{2}{|x_i'|}\varepsilon \\
\Longrightarrow |\prod_{m \neq i,j} x_m'|^2 & \le& (2+\varepsilon)^2 \quad \hbox{since} ~~|x_i'|>1 \\
\Longrightarrow |\phi(\gamma_j)|=|\prod_{m \neq i,j} x_m'|& \le& 2+\varepsilon.
\end{eqnarray*}
Hence $\gamma_j=[v';\{i,j\}] \in \A_{\phi}(K)$ if $\varepsilon$ is chosen to be $\min \{K-2, 2^{1/(n-2)}-1\}$.

To recap, we have shown that either $\alpha$ meets some $\gamma_m=[v';\{i,m\}] \in \A_{\phi}(2)$, or  $\alpha$ meets $\gamma_j=[v';\{i,j\}] \in \A_{\phi}(K)$ which proves the claim. To continue, if the tail of $\alpha$ is eventually contained in some $\gamma\in \A_{\phi}(K)$, then by Lemma \ref{lem:geodesicalong}, we must have either $\phi(\gamma) \in[-2,2]$ (in which case we are done), or $\sigma(\gamma)=\mu \neq \pm 2$. In the latter case, it is easy to see that since the values of $y_m$ around $\gamma$ approach zero,  $\alpha$ meets infinitely many $\gamma_k \in \A_{\phi}(K)$. On the other hand, if $\alpha$ is not eventually contained in any alternating geodesic $\gamma \in \A$, then by Lemma \ref{lem:2away} (iii), we see that  if the edge $e_m$ leaves $\gamma \in \A_{\phi}(K)$, then $\alpha$ meets a new $\gamma' \in \A_{\phi}(K)$ at the head of $e_m$. Hence, in this case, $\alpha$ meets infinitely many $\gamma \in \A_{\phi}(K)$ as well.
\end{proof}

A direct consequence of the above and the definition of the Bowditch conditions gives:

\begin{cor}\label{cor:sinkexist}
If $\phi \in \B$, then there is no infinite geodesic ray in $\vecDelta$ consisting of a sequence of edges
 such that  each edge is $\phi$-directed towards the next. Hence $\vecDelta$ contains a sink.
\end{cor}

 \subsection{Attracting subtrees}\label{def:attracting}

 \begin{defn}Given $\phi \in \Hur$, a subtree $T$ (possibly a single vertex) of $\vecDelta$ is said to be $\phi$-{\it attracting}
 if  every edge $\vec e \in E(\vecDelta)$ not contained in $T$ is directed decisively towards $T$.
\end{defn}
 The aim of this and the next subsections is show that for every $\phi\in \B$ there is finite subtree $T_{\phi} \subset \vecDelta$ which is $\phi$-attracting.
We first consider the case where $\A_{\phi}(K)=\emptyset$.

 \begin{prop}\label{prop:sink}
 If $\phi \in \B$ is such that $\A_{\phi}(K)=\emptyset$ for some $K>2$,
 then there is a unique sink in $\vecDelta$ which is $\phi$-attracting.
 \end{prop}

 \begin{proof} 
 By Corollary \ref{cor:sinkexist}, there is a sink $v$ in $\vecDelta$.
 Then $v$ is $\phi$-attracting; otherwise, some $\phi$-directed edge $\vec{e}$ is directed away from $v$, and the path in $\Delta$ connecting $v$ and the head of $\vec{e}$  contains a  fork. By Lemma \ref{lem:2away} (i),
   $\A_{\phi}(2) \neq \emptyset$, a contradiction.

 \end{proof}

\smallskip

 \subsection{Attracting subtree $T_{\phi}$}\label{ss:attrsubtree}


For a general $\mu$-Hurwitz map $\phi\in \B$ where $\A_{\phi}(K)\neq\emptyset$ for some some fixed $K >2$, we will construct a finite connected subtree $T_{\phi}$ which will be $\phi$-attracting. This will consist of the union of non-empty,  finite, connected sub-intervals $J(\gamma)$ of $\gamma \in \A_{\phi}(K)$ which will have the following desired properties:
\begin{enumerate}
\item Every $e$ which lies in the intersection of two alternating geodesics $\alpha, \beta \in \A_{\phi}(K)$ is contained in  $T_{\phi}$.
\item If $\gamma \in \A_{\phi}(K)$ and $J(\gamma)\subset \gamma$ is the non-empty, connected finite sub-interval constructed, then all edges on $\gamma$ not contained in $J(\gamma)$ are $\phi$-directed towards $J(\gamma)$.
\end{enumerate}

It is not difficult to see that it is always possible to construct such a subtree, by Lemma \ref{lem:geodesicalong}. We can do this more systematically by first constructing the  function $H_{\mu}:\CC^{n-2} \rightarrow \CC \cup \{\infty\}$ as follows:

 Let $\mu \in \CC$. For $(x_1, \cdots, x_{n-2}) \in \CC^{n-2}$, let $x:=\prod_{i=1}^{n-2} x_i$, and $\lambda\in \CC$ be such that $\lambda+\lambda^{-1}=x$ and $|\lambda| \ge 1$, and $\sigma=\sum_{i=1}^{n-2} x_i^2$. Define

\begin{equation}
H_{\mu}(x_1, \cdots,x_{n-2})=
\begin{cases}
\infty, \quad \hbox{if}~~ x \in [-2,2]~~\hbox{or}~~\sigma=\mu \\
 \sqrt{\big|\frac{\sigma - \mu}{x^2 - 4} \big|}\frac{2|\lambda|^2}{|\lambda|-1} \quad \hbox{otherwise}.
\end{cases}
\end{equation}


The function $H_{\mu}$ has the following property (see \cite[pp. 780-781]{tan-wong-zhang2008advm}):

\begin{lem}\label{lem:interval} Let $\mu \in\mathbb C$, and $(x_1,\cdots,x_{n-2}) \in \mathbb{C}^{n-2}$ such that $x:=\prod_i x_i \notin [-2,2]$ and $\sigma:= \sum_i x_i^2 \neq \mu$. Let $\{y_m\}_{m\in\mathbb Z}$ be a sequence of complex numbers such that $y_{m-1}+y_{m+1}=x y_m$ and $\mu=\sigma + y_{m}^2+y_{m+1}^2- x y_{m} y_{m+1}$ for all $m \in \mathbb Z$.
Then there exists $m_1< m_2\in\mathbb Z$  such that $|y_m|\le H_{\mu}(x_1,\cdots,x_{n-2})$ exactly when $m_1 \le m \le m_2$,
 and $|y_m|$ is strictly decreasing on $(-\infty, m_1]\cap \mathbb Z$ and strictly increasing on $[m_2,\infty) \cap \mathbb Z$.
\end{lem}

We can translate this to the following result for $\gamma\in \A$, by first adapting $H_{\mu}$ to a function on $\A$. Let  $\phi \in \Hur_{\mu}$ and $\gamma=\bigcap_{k=1}^{n-2}X_k \in \A$ where $X_k \in \C$. Define
$$H_{\mu}(\gamma)=H_{\mu}(x_1, \cdots, x_{n-2}).$$
Adopting the notation from \S \ref{ss:growth} and the results there, and applying Lemm\ref{lem:interval},  we get:
\begin{lem}\label{lem:directedin}
Let $\phi \in \Hur_{\mu}$ and $\gamma \in \A$ such that $\phi(\gamma) \notin [-2,2]$ and $\sigma(\gamma) \neq \mu$. Let $\{e_m\}_{m \in \Z}$ be the sequence of edges of $\gamma$ connecting $v_m$ to $v_{m+1}$ and $Y_m \in \C$ be such that $Y_m \cap \gamma=\bar{e}_m$. Then there is a non-empty interval $I(\gamma) \subset \gamma$ from $v_s$ to $v_t$ ($s<t$) such that $|y_k| \le H_{\mu}(\gamma)$ if and only if $s \le k \le t$. Furthermore, all the edges on $\gamma$ not contained in $I(\gamma)$ are $\phi$-directed towards it.
\end{lem}
Note that if  $\phi(\gamma) \in [-2,2]$ or $\sigma(\gamma) = \mu$, $H(\gamma)=+\infty$ and we define $I(\gamma)=\gamma$  in this case.

Our strategy now is to take $T_{\phi}$ to be the union of $I(\gamma)$ for all $\gamma \in \A_{\phi}(K)$, where $K>2$ is fixed. However, in order to make sure that the desired property (1) in the beginning of this subsection is satisfied, we have to modify our function $H_{\mu}$ slightly to ensure that all edges $e \subset \alpha \cap \beta$ where $\alpha, \beta \in \A_{\phi}(K)$ are contained in $T_{\phi}$. We define
\[
H_{\mu}^* (\gamma) = \max \Big\{ H_{\mu}(\gamma), \max_{1\le i \le n-2}\{\frac{K}{|\prod_{j \neq i} x_j|}\} \Big\},
\]
and let
\[
J_{\phi}(\gamma)=\{e_m \in E(\gamma): |y_m| \le H_{\mu}^*(\gamma)\}.
\]
We can easily check that with this adjustment, $J(\gamma)$ has the same properties as $I(\gamma)$ in Lemma \ref{lem:directedin} and furthermore, property (1) is satisfied. That is, if $K>2$ and $\alpha, \beta \in \A_{\phi}(K)$ and $e \subset \alpha \cap \beta$, then $e \subset J(\alpha) \cap J(\beta)$.
With this, if $\phi \in \B_{\mu}$, define
\[T_{\phi}= \bigcup_{\gamma \in \A_{\phi}(K)}J(\gamma).
\]

\begin{lem}\label{lem:finiteT}
Let $\phi \in \B_{\mu}$ and suppose $\A_{\phi}(K) \neq \emptyset$ for some $K>2$. Then  $T_{\phi}= \bigcup_{\gamma \in \A_{\phi}(K)}J(\gamma)$ is a finite, connected subtree of $\Delta$.
\end{lem}

\begin{proof}
The set $\A_{\phi}(K)$ is non-empty and finite by assumption, and for $\gamma \in \A_{\phi}(K)$, $\phi(\gamma) \not\in [-2,2]$ (otherwise it violates Bowditch condition (B1)), and $\phi(\gamma) \neq  \mu$ (otherwise condition (B2) will be violated, by Lemma \ref{lem:geodesicalong}(i), as argued in the proof of Lemma \ref{lem:escpath}). Hence for each $\gamma \in \A_{\phi}(K)$, $J(\gamma)$ is a (non-empty) finite interval. The fact that the (finite) union of these subintervals $J(\gamma)$ is connected follows from the edge-connectedness of $\A_{\phi}(K)$, the connectedness of $J(\gamma)$, and property (1) that all edges $e \subset \alpha \cap \beta$ where $\alpha, \beta \in \A_{\phi}(K)$ are contained in $J(\alpha)$ and $J(\beta)$.
\end{proof}

\begin{lem}
With the same assumptions and notation as the previous lemma, the tree $T_{\phi}$ satisfies:
\begin{enumerate}
\item If $e \in E(\Delta)$ meets $T_{\phi}$ at a single vertex, then $e$ is $\phi$-directed towards $T_{\phi}$;
\item If $v \in V(\Delta)$ is a fork or a sink, then it lies in $T_{\phi}$. Hence all vertices not in $T_{\phi}$ are merges;
\item All edges $e \in E(\Delta)\setminus E(T_{\phi})$  are $\phi$-directed decisively towards $T_{\phi}$.
\end{enumerate}
\end{lem}

\begin{proof}

(1) If $e \subset \gamma \in \A_{\phi}(K)$, but $e \not\subset J(\gamma)$, then $e$ is directed towards $T_{\phi}$ by Lemma \ref{lem:directedin} . Otherwise, let $e \in E_{\{k\}}(\Delta)$ with end vertices $v$ and $w$ where $v \in T_{\phi}$. Since $v \subset T_{\phi}$, $\gamma:=[v;\{i,j\}] \in \A_{\phi}(K)$ for some $i,j \neq k$, see Figure \ref{fig:theedgee}. Now if $e$ is $\phi$-directed away from $v$, then $\gamma':=[w;\{i,j\}] \in A_{\phi}(K)$ by Lemma \ref{lem:2away}(iii). Then $T_{\phi}\supset J(\gamma')$ and the connectedness of $T_{\phi}$ implies $e \subset T_{\phi}$, a contradiction. Hence in this case $e$ is also $\phi$-directed towards $T_{\phi}$.


\begin{figure}[ht]
\centering
\def\svgwidth{.61\columnwidth}
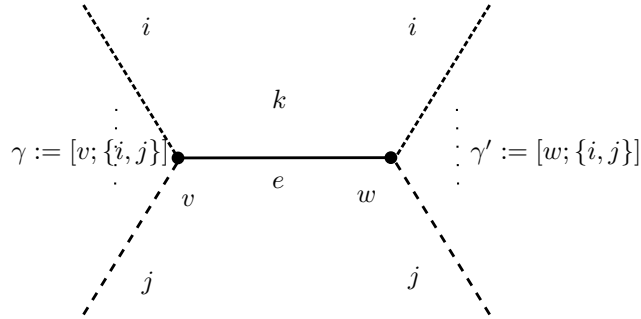
\caption{The edge $e$ where $\bar e \cap T_{\phi}=v$}
\label{fig:theedgee}
\end{figure}

(2) If $v \in V(\Delta)$ is a fork, then $\alpha=[v;\{i,j\}] \in \A_{\phi}(2)$ where the edges $e_i$ and $e_j$ incident to $v$ point away from $v$. Then $e_i, e_j \subset J(\alpha) \subset T_{\phi}$ by Lemma \ref{lem:directedin}. Hence $v \in V(T_{\phi})$. If $v$ is a sink  which is not contained in $T_{\phi}$ then by part (1), the geodesic from $v$ to $T_{\phi}$ contains a fork $v'$ which is not in $T_{\phi}$, again a contradiction.
 Hence all vertices not in $T_{\phi}$ are merges.

(3) This follows immediately from the first two parts.
\end{proof}

The two lemmas above imply:

\begin{prop}\label{prop:attrsubtree}
If $\phi \in \B$ is such that $\A_{\phi}(K)$ is nonempty and finite for some $K>2$, then $T_{\phi}$ defined above is nonempty, finite and $\phi$-attracting.
\end{prop}

 \subsection{Fibonacci growth}
 Given a function $f:\A \rightarrow {\mathbb R}_{\ge 0}$ and $\A' \subset \A$, we say that $f$ has {\it Fibonacci growth} on $\A'$
 if it has both lower and upper Fibonacci bounds on $\A'$. That is,  there exists a constant $K>0$ such that
\[
\frac{1}{K} F(\gamma) \le f(\gamma) \le K F(\gamma), \quad \hbox{ for all but finitely many }\gamma \in \A'.
\]

\noindent Equivalently,  there exists  constants $K', C >0$  such that
\[
\frac{1}{K'} F(\gamma) -C \le f(\gamma) \le K' F(\gamma) + C, \quad \hbox{ for all }\gamma \in \A'.
\]
 The left inequality is for lower Fibonacci bound, the right inequality for upper Fibonacci bound. The lower Fibonacci bound will be the one that is relevant for our later analysis.

We say that $f$ has {\it Fibonacci growth} if it has Fibonacci growth on all of $\A$. Our aim will be to show that if $\phi \in \B$,
then $\log^+ |\phi|:=\max\{ 0, \log|\phi(\gamma)| \}$ has Fibonacci growth on all of $\A$.

\subsubsection{Upper Fibonacci bounds}
The following proposition shows that the function $\log^+ |\phi|$ always has an upper Fibonacci bound.
\vspace{5pt}

\begin{prop}\label{prop:upperfibonacci}
For any $\mu$ and $\phi \in \Hur_{\mu}$, the induced function $\log^+ |\phi|$ has an upper Fibonacci bound on $\A$.
\end{prop}

\begin{proof}
By \cite{bowditch1998plms} Lemma 2.1.1(i), it suffices to show that whenever $\alpha, \beta, \gamma \in \A$ meet at a vertex $v$ and pairwise share an edge,  and $d(\beta), d(\gamma) <d(\alpha)$, then
\[
\log^+ |\phi(\alpha)| \le \log^+ |\phi(\beta)| + \log^+ |\phi(\gamma)| + \log^+|\mu| + \log^+ 2n.
\]
 From the definition of $\phi$, with some manipulation, the above inequality holds if
\begin{equation}
\label{upperFinineq}
\log^+ |x_n| \le \log^+ |x_1| + \cdots + \log^+ |x_{n-1}| + \log^+|\mu| + \log^+ 2n
\end{equation}
for any n-tuple $\mathbfit{x}=(x_1, \cdots, x_n)$ satisfying the equation $H(\mathbfit{x}) = \mu$.
\vspace{5pt}

In fact, if $|x_n| \le n |x_i|$ for some $i= 1, \cdots, (n-1)$, then \eqref{upperFinineq} holds naturally. Then suppose $|x_n| \ge n|x_i|$ for all $i= 1, \cdots, (n-1)$, we have
\begin{align*}
|\mu|+ |x_1 x_2 \cdots x_n| & \ge |x_n|^2 - |x_1|^2 - \cdots - |x_{n-1}|^2\\
& =  |x_n|^2/n + (|x_n|^2/n - |x_1|^2) + \cdots + (|x_n|^2/n - |x_{n-1}|^2)\\
& \ge |x_n|^2/n.
\end{align*}
Hence $|x_n|^2 \le 2n |x_1 x_2 \cdots x_n|$ if $|\mu| \le |x_1 x_2 \cdots x_n|$, or otherwise, $|x_n|^2 \le 2n |\mu|$. Thus \eqref{upperFinineq}, and the proposition follows.
\end{proof}

\subsubsection{Lower Fibonacci bounds}
 \begin{prop}\label{prop:lowerfibonacci}
 If $\phi \in \B$ then $\log^{+}|\phi|$ has a lower Fibonacci bound on $\A$.
 \end{prop}

 \begin{proof}
 Since there is a $\phi$-attracting finite subtree $T$ of $\vecDelta$,
 it suffices to prove that for each $\vec{e}_0 \in C(T_{\phi})$,  $\log^{+}|\phi|$ has a lower Fibonacci bound on $\A^{0-}(\vec{e}_0)$. By \eqref{eqn:Fibonacciv}, we may assume that the head of $\vec{e}_0$ is the root $v_0$.


 \vskip 6pt

 Case 1. $\A^{0-}_{\phi}(\vec{e}_0) \cap \A_{\phi}(2)=\emptyset$.

 \vskip 6pt

 Let $c=\min\{\log\frac12|\phi(\alpha)| \mid \alpha \in \A^{0}_{\phi}(e_0)\}$. Then $c>0$ and, for $\alpha \in \A^{0}(e_0)$,
 $$ \log \frac{|\phi(\alpha)|}{2} \ge c = c F_{v_0}(\alpha). $$
 Given $\alpha \in \A^{-}(\vec{e}_0)$, Let $e$ be the edge incident to $\alpha$ on the
 geodesic path from $v_0$ to $\alpha$ and  $\alpha',\beta,\gamma \in \A$ be the paths incident to $e$ as defined in \ref{ss:relgeodesics}, see figure \ref{fig:edgerelation2}. In particular, $d(\alpha'), d(\beta), d(\gamma) < d(\alpha)$, $|\phi(\alpha)|\ge |\phi(\alpha')|$ (since $\vec{E}_{\phi}(e)$ is directed towards $\alpha'$) and $\phi(\alpha)+\phi(\alpha')=\phi(\beta) \phi(\gamma)$.

 Thus we have $2|\phi(\alpha)|\ge|\phi(\beta)\phi(\gamma)|$.
 By induction on the distance of $\alpha$ from $v_0$, we have
 \begin{eqnarray*}
 \log\frac{|\phi(\alpha)|}{2}
 &\ge& \log\frac{|\phi(\beta)|}{2} + \log\frac{|\phi(\gamma)|}{2} \\
 &\ge& c F_{v_0}(\beta) + c F_{v_0}(\gamma) = c F_{v_0}(\alpha).
 \end{eqnarray*}

 \vskip 6pt

 Case 2. $\A^{0-}_{\phi}(\vec{e}_0) \cap \A_{\phi}(2) \neq \emptyset$.

 \vskip 6pt

 We may assume $\A_{\phi}^{0-}(\vec{e}_0) \cap \A_{\phi}(2)=\{\beta\} \subset \A_{\phi}^{0}(e_0)$.
 Suppose the ray $\beta \cap \Delta^{-}(\vec{e_0})$ consists of a sequence of edges $\{e_m\}_{m=1}^{\infty}$ where $\bar{e}_m=\beta \cap Y_m$ for $m \ge 0$, and $v_m=\bar{e}_m \cap Y_{m-1}$ for $m \ge 1$.
 Then each of the directed edges $\{\vec{e}_m\}_{m=1}^{\infty}$ is directed towards $v_m$, see Figure \ref{fig:halfalternating}.
 Let $\vec{\varepsilon}_m^{\,j}, j=1,\cdots,n-2$ be the other $n-2$ edges incident to $v_m$ endowed with $\phi$-directed  arrow which are directed towards $v_m$, and $\gamma_m^{j}$ be the alternating geodesic containing $v_m$, $e_m$ and $\varepsilon_m^{\,j}$ for $j=1,\cdots, n-2$.
 We see that $F_{v_0}(\gamma_m^{j})=m+1$ and $F_{v_m}(\gamma_m^{j})=1$.
 By the exponential growth of $|y_m|$ (and hence of $|\phi(\gamma_m^{j})|$) according to Lemma \ref{lem:geodesicalong} (2), there is some $c>0$ such that
 $$ \log \frac{|\phi(\gamma_m^{j})|}{2} \ge 4 c (m+1) $$
 for all $j \in [n-2]$ and $m \ge 1$.

\begin{figure}[ht]
\centering
\def\svgwidth{\columnwidth}
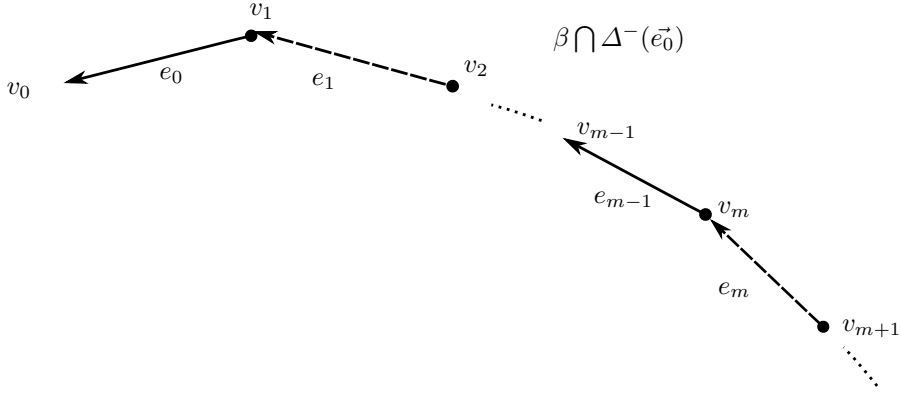
\caption{Part of the alternating geodesic $\beta$}
\label{fig:halfalternating}
\end{figure}

 By Case 1, for any $j \in [n-2]$ and $\alpha \in \A^{0-}(\vec{\varepsilon}_m^{\,j})$, we have
 \begin{eqnarray*}
 \log\frac{|\phi(\alpha)|}{2}
 &\ge& \min\Big\{\log\frac{|\phi(\alpha')|}{2} \mid \alpha' \in \A^{0}(\varepsilon_m^{j})\Big\} F_{v_m}(\alpha) \\
 & = & \min\Big\{\log\frac{|\phi(\gamma_{m-1}^{j})|}{2},\log\frac{|\phi(\gamma_m^{j})|}{2}\Big\} F_{v_m}(\alpha).
 \end{eqnarray*}
 It follows from the fact $F_{v_0}(\alpha) \le 2 (m+1) F_{v_m}(\alpha)$ that
 $$ \log^{+}\frac{|\phi(\alpha)|}{2} \ge 4 c m F_{v_m}(\alpha) \ge 2c \frac{m}{m+1} F_{v_0}(\alpha) \ge c F_{v_0}(\alpha).$$
 Since $\A^{0-}(\vec{e}_0) = \{\beta\} \bigsqcup \big(\bigcup_{j=1}^{n-2} \bigcup_{m \ge 1 }\A^{0-}(\vec{\varepsilon}_m^{\,j}) \big)$, $\log^{+}|\phi|$ has a lower Fibonacci bound on $\A^{0-}(\vec{e_0})$. This concludes the proof of Proposition \ref{prop:lowerfibonacci}.
 \end{proof}

%

 \subsection{Enlarged attracting subtree }

Let $\gamma \in \A_{i,j}$; for convenience, assume $i=n-1$, $j=n$. As in \S \ref{ss:growth}, we have $\gamma=\bigcap_{k=1}^{n-2} X_k $ where $X_k \in \C_k$. Let the sequence of edges of $\gamma$ be  $\{e_m\}_{m \in \Z}$, and $Y_m \in \C$ such that   $Y_m \bigcap \gamma=\bar{e}_m$.
Given $\phi\in\Hur_{\mu}$ and $t \ge 0$,  we define
\[
H_{\mu}^t (\gamma) = \max \Big\{ H_{\mu}(\gamma) + t, \max_{1\le i \le n-2}\{\frac{2+ t}{|\prod_{j \neq i} x_j|}\}  \Big\},
\]
and let
\[
J_{\phi}(\gamma;t)=\bigcup \bar{e}_m \subset \gamma\]
where \[ |y_m| \le H_{\mu}^t(\gamma);
\]
otherwise, $J_{\phi}(\gamma;t)=\gamma$. Now
we define a subtree $T_{\phi}(t)$ of $\Delta$ as the union of $J_{\phi}(\gamma;t)$  as $\gamma$ varies in $\A_{\phi}(2+t)$. Precisely, $e\in E(\Delta)$ is in $T_{\phi}(t)$ if and only if $e=\gamma \cap Y$ for some $\gamma\in\A$ and $Y\in \C$
 such that either $\phi(\gamma) \in [-2,2]$, or else $|\phi(\gamma)| \le 2+t$ and $|y| \le H_{\mu}^t (\gamma)$.

If for $\phi \in \B$, $\A_{\phi}(K)$ is nonempty for some $K >2$, then $T_{\phi}(t)$ is a nonempty, finite, $\phi$-attracting subtree of $\vecDelta$ for any $t \ge K-2$. Moreover,

\begin{prop}\label{prop:enlattrsubtree}
For any fixed $t > 0$, $\phi\in \B$ if and only if $T_{\phi}(t)$ is finite.
\end{prop}

The proof is straightforward, and follows directly from the definition of $\B$ and $T_{\phi}(t)$. Note that $T_{\phi}(t)=\emptyset$ for some positive $t$ implies $\A_{\phi}(2+t) = \emptyset$, thus the Bowditch conditions (B1) and (B2) are satisfied.

 \subsection{Openness of $\B$}

 With the help of the enlarged attracting subtrees described in the previous subsection,
 we show that the Bowditch sets $\B$ and $\B_{\mu}$ are open subsets of $\Hur$ and $\Hur_{\mu}$, respectively.

 \begin{thm}\label{thm:open}
 The Bowditch set $\B$ is open in $\Hur$, and $\B_{\mu}$ is open in $\Hur_{\mu}$.
 \end{thm}

 \begin{proof}
 The proof is essentially the same as the proofs for \cite[Theorem 3.16]{bowditch1998plms} 
 with suitable change of notation. We reproduce it here.
 Fix any $\mu$-Hurwitz map $\phi \in \B$, we have $\phi(\gamma) \notin [-2,2]$ for all $\gamma\in \A$ and $T_{\phi}(t)$ is a non-empty finite subtree for $t \ge M(\mu)-2$. For any $\phi'$ in $N_{\varepsilon}(\phi)$, a small open neighborhood of $\phi$, we can assume $\phi'$ is a $\mu'$-Hurwitz map where $\mu'\in N_{\varepsilon'}(\mu)$ and $\varepsilon'$ is close to 0. Denote by $\phi', \sigma'$ the induced functions on $\A$, and let
 \[
 M=\sup_{\mu'' \in N_{\varepsilon'}( \mu)}\{ M(\mu'')\} < +\infty,
 \]
 for sufficiently small $\varepsilon'$.
 Let $T_{\phi'}(t)$ be the constructed subtree for any $t \ge 0$.

Fix $t_1 > t_0 > M - 2$ and choose sufficiently large $t_2 > t_1$ such that $T_{\phi}(t_2)$ contains $T_{\phi}(t_1)$ in its interior. That is, $T_{\phi}(t_2)$ contains $T_{\phi}(t_1)$, together with all the edges of the circular set $C(T_{\phi}(t_1))$. On the one hand, $T_{\phi'}(t_0) \supset T_{\phi}(M-2) \neq \emptyset$ if $\phi'$ is sufficiently close to $\phi$. On the other hand, we see that if $\phi'$ is sufficiently close to $\phi$, then $T_{\phi'}(t_0) \bigcap T_{\phi}(t_2) \subset T_{\phi}(t_1)$. 
For an arbitrary $e \in T_{\phi}(t_2)\backslash T_{\phi}(t_1)$, we claim that $e \notin T_{\phi'}(t_0)$ if $\phi'$ is sufficiently close to $\phi$. In fact, suppose $e=X_1\cap\cdots\cap X_{n-1}$ where $X_1,\cdots,X_{n-1}\in \mathcal{C}$. Write $\gamma_i=X_1\cap\cdots\cap\hat{X}_i\cap\cdots\cap X_{n-1}$ for $i=1,\cdots,n-1$.
Since $e\not \in T_{\phi}(t_1)$, we have for each $i=1,\cdots,n-1$, either $|\phi(\gamma_i)| > 2+ t_1$ or $|\phi(X_i)| > H_{\mu}^{t_1}(\gamma_i)$. If $\phi' \in \Hur$ is sufficiently close to $\phi$, then for each $i=1,\cdots,n-1$, either $|\phi'(\gamma_i)| > 2+ t_1 > 2+ t_0$ or 
$|\phi'(X_i)| > H_{\mu}^{t_1}(\gamma_i) > H_{\mu'}^{t_0}(\gamma_i)$; therefore $e \not \in T_{\phi'}(t_0)$. This proves the claim since there are only finitely many edges in $T_{\phi}(t_2)\backslash T_{\phi}(t_1)$.
By the connectedness of $T_{\phi'}(t_0)$, we have $T_{\phi'}(t_0) \subset T_{\phi}(t_2)$. 
Thus $T_{\phi'}(t_0)$ is finite, which implies that $\phi' \in \B$ if $\phi'$ is sufficiently close to $\phi$.

In the above we have shown that $\B$ is open in $\Hur$; the claim that $\B_{\mu}$ is open in $\Hur_{\mu}$ follows easily. This proves Theorem \ref{thm:open}.
 \end{proof}

\subsection{The diagonal slice}\label{ss:diagonal}
The results in the above sections allow us to write a computer program to determine  approximations of the set $\D$, consisting of those $\mathbfit{x} \in \D$, with associated Hurwitz map $\phi$  where the attracting tree $T_{\phi}$ is within a specified size. A natural slice to consider is the diagonal slice, namely the set of $x \in \CC$ such that $(x,\cdots,x) \in \D$ (see \cite{STY} for the case where $n=3$ where this slice was compared to the Schottky slice). We are grateful to Yasushi Yamashita who has written a program which draws the diagonal slices in the case $n=3$ and $n=4$ given in Figure \ref{fig:diagslices}. As can be seen, there are many interesting features to these slices, most parts of the boundary are fractal  with inward pointing cusps, much like the slices of various subsets of the quasi-Fuchsian space of a surface. In the case of $n=3$ there appears to be a  rather mysterious ``tail'' namely the real line segment $[-2, s)$ where $-2<s<-1$ which does not lie in $\D$ but which appears to lie in its closure $\overline \D$.

The next proposition shows that the diagonal slice contains the exterior of a disk centered at the origin with radius $r=2^{1/(n-2)}$, in particular, it implies our assertion that $\D$ is non-empty.

\begin{prop}\label{prop:diagonal}
Let $n \ge 3$ be fixed. If $|x|^{n-2}>2$, then $(x,\cdots,x) \in \CC^n$ is in $\D$.
\end{prop}

Note that if $x \in \RR$ and $|x|^{n-2} \le 2$, then $(x, \cdots, x) \notin \D$ since condition (B1) is not satisfied.

\begin{proof} Suppose $|x|^{n-2}>2$ and let $\phi$ be the Hurwitz map  induced from $\mathbfit{x}=(x,\cdots, x)$. We will show that there exists $K>2$ such that $\A_{\phi}(K)=\emptyset$, hence the Bowditch conditions are satisfied. For $\gamma \in \A$, recall that $d(\gamma)$ denotes the distance of $\gamma$ from the root vertex $v_0$.  If $d(\gamma)=0$, $|\phi(\gamma)|=|x^{n-2}|>2$.
We claim that all the edges $e_j$ adjacent to $v_0$ are $\phi$-directed towards $v_0$. This follows from
$$|x^{n-1}-x|=|x||x^{n-2}-1|>|x|$$ since
 $$|x^{n-2}-1| \ge |x|^{n-2}-1>1.$$
 Choose $K$ such that $2<K<|x|^{n-2}$.
We claim that $\A_{\phi}(K)=\emptyset$. Otherwise, there exists $\gamma \in \A_{\phi}(K)$ such that $d(\gamma)=t$ is minimal, and from the above discussion, $t=d(\gamma)\ge 2$. Consider the geodesic path from $v_0$ to $\gamma$. As in the proof of Propostion \ref{prop:e-c}, the two edges at the ends of this path are $\phi$-directed outwards. This implies that some vertex on this path is a fork, and from Lemma \ref{lem:2away} (i) this implies that there is a path $\gamma' \in \A_{\phi}(2) \subset \A_{\phi}(K)$ with $d(\gamma')<t$ which contradicts the minimality of $t$. Hence $\A_{\phi}(K)=\emptyset$ as claimed.
\end{proof}

\subsection{The complement of $\D$ in $\CC^n$}\label{ss:complement}
 We have seen that $\D$ is non-empty, here we show that the closure $\overline{\D}$ is not all of $\CC^n$. In particular, $\mathbf{0} \not\in \overline{\D}$.
 Let $B_{\varepsilon}(\mathbf{0})$ denote the ball of radius $\varepsilon$ about $\mathbf{0}$ in $\CC^n$.
 \begin{thm}\label{thm:complementset}
 There exists a positive constant $\varepsilon$ such that $B_{\varepsilon}(\mathbf{0})\cap \D=\emptyset$.
 \end{thm}

 We first note that for sufficiently small $\varepsilon$,  if $\mathbfit{a}\in B_{\varepsilon}(\mathbf{0})$ and the corresponding Hurwitz map $\phi \in \Hur_0$, then the same computations as in \cite{ng-tan2007osakajm} will show that we can obtain either an infinite descending path in $\vec{\Delta}_{\phi}$, or $\phi(\gamma) \in (-2,2)$ for some $\gamma\in \A$. In either case, $\phi \notin \B$. If $\mu \neq 0$, the argument is a little more delicate, but we can show the following. 

 Suppose that $\phi \in \Hur_{\mu}$ and $\gamma =\bigcap_{k\neq i,j}X_k \in \A_{ij}$ where $X_k \in \C_k$. Let  $\| \gamma \|_{\phi}:=\max_{k\neq i,j}|x_k|$.

 \begin{lem}
 Let $\gamma =\bigcap_{k\neq i,j}X_k \in \A_{ij}$ consists of the sequence of edges $\{e_m\}$ and let $Y_m\in \C$ be such that $Y_m \cap \gamma=\bar{e}_m$ as in \S \ref{ss:growth}.
 There exist positive constants $\varepsilon_1$ and $\varepsilon_2$ such that if  $\phi \in \Hur_{\mu}$ where $\mu <\varepsilon_1$, $\phi(\gamma) \notin(-2,2)$, and $\|\gamma\|_{\phi}<\varepsilon_2$,  then there exists $s\in \Z$ such that $|y_s|, |y_{s+2}|<\varepsilon_2$.
 \end{lem}

 \begin{proof}
 The proof is similar to the proof the main theorem  in \cite{ng-tan2007osakajm} but slightly more technical,  we omit the details.
 \end{proof}

 Now if $\mathbfit{a} \in B_{\varepsilon}(\mathbf{0})$ for sufficiently small $\varepsilon$, with induced Hurwitz map $\phi$ and $\gamma =\bigcap_{k\neq i,j}X_k \in \A_{ij}$ contains the root $v_0$, then $\mu <\varepsilon_1$ and  $\|\gamma\|_{\phi}<\varepsilon_2$. Then either $\phi(\gamma) \in (-2,2)$ in which case $\phi \notin \B$, or by the lemma above, by replacing one of the $X_k$, $k \neq i,j$  by $Y_s$ or $Y_{s+2}$, we deduce that there are at least two $\alpha, \beta \in \A$ which are edge-connected to $\gamma$, but not to each other satisfying $\|\alpha\|_{\phi}<\varepsilon_2$, $\|\beta\|_{\phi}<\varepsilon_2$. Furthermore, in general, because $\alpha$ and $\beta$ are not edge-connected, either $d(\alpha)>d(\gamma)$ or $d(\beta)>d(\gamma)$. We can then inductively produce either an infinite sequence $\gamma_m$ with
 $\|\gamma_m\|_{\phi}<\varepsilon_2$ or some $\delta \in \A$ with $\phi(\delta) \in (-2,2)$. In either case, $\phi \notin \B$ which proves Theorem \ref{thm:complementset}.

 \subsection{Properly discontinuous action of $\Gamma_n$ on $\D$}

 \begin{thm}\label{thm:properdisc}
 The action of $\Gamma_n$ on $\D$ (or $\D_{\mu}$) is properly discontinuous.
 \end{thm}

 \begin{proof}
 The proof is essentially the same as that of Theorem 2.3 in \cite{tan-wong-zhang2008advm}.
 We show that for any given compact subset $\mathcal K$ of $\D$,
 the set $\{g\in\Gamma_n \mid (g\mathcal K)\cap\mathcal K\neq\emptyset\}$ is finite.
 Suppose not, then there exists a sequence of distinct elements $(g_m)_{m\in\mathbb N}$ in $\Gamma_n$ and $\mathbfit{a}_m \in \mathcal K$ such that $g_m(\mathbfit{a}_m) \in\mathcal K$ for all $m\in\mathbb N$. Since $\mathcal K$ is compact, by passing to a subsequence, we may assume that
 $\mathbfit{a}_m \rightarrow \mathbfit{a}\in\mathcal K$. Let $\phi=\phi_{\mathbfit{a}}$ and $\phi_m=\phi_{\mathbfit{a}_m}$, $m\in\mathbb N$.
 As in the proof of Theorem \ref{thm:open}, there exist $t_2>t_1>0$ such that the tree $T_{\phi}(t_1)$ is finite,
 and for all sufficiently large $m$, $T_{\phi_m}(t_1)$ is contained in the interior of $T_{\phi}(t_2)$.
 This implies that when $m$ is sufficiently large, all $\log^{+}|\phi_m|$ can have the same constant appearing in the lower Fibonacci bounds.
 Hence $g_m(\mathbfit{a}_m)$ is far away from $\mathcal K$ for sufficiently large $m$, a contradiction. This proves Theorem \ref{thm:properdisc}.
 \end{proof}

 \section{Identities and proofs}\label{s:proof}

 This section is devoted to the proof of the identities. We start with some preliminary results on convergence of various infinite sums.

 \subsection{Convergence of  infinite sums}
 \begin{prop}\label{prop:lowerfibconvergence}
 Given a function $f:\A\rightarrow \mathbb R_{>0}$, if $\log^{+}f$ has a lower Fibonacci bound on $\A$,
 then the infinite sum $\sum_{\gamma\in\A}(f(\gamma))^{-t}$ converges absolutely for all $t>0$.
 \end{prop}

 \begin{proof}
 By definition, there exists a positive constant $K$ such that, for all but finitely many $\gamma \in \A$,
 $\log f(\gamma) \ge K^{-1} F(\gamma)$. Hence
\[
f(\gamma) \ge  \exp (K^{-1} F(\gamma)) >\frac{(K^{-1} F(\gamma))^{N}}{N!} \quad
\hbox{ for all} \quad N \in \mathbb N.
\]
 Choose $N \in \mathbb N$ so that $Nt >n-1$. Then
 $$ \sum_{\gamma\in\A}(f(\gamma))^{-t} \le C \sum_{\gamma\in\A}(F(\gamma))^{-Nt} < +\infty $$
 for some constant $C>0$, where the final inequality follows from Proposition \ref{prop:fibconvergence}.
 \end{proof}

Recall (\S \ref{ss:extendedhurwitz}) that a Hurwitz map $\phi: \C \rightarrow \CC$ can be extended to $T^{|k|}(\Delta) \rightarrow \CC$ for $0 \le k \le n-2$, in particular to $T^{|0|}(\Delta)=V(\Delta) \rightarrow \CC$. We have:

\begin{prop}\label{prop:conv}
 Given $\phi \in \B$ and $t>0$, the infinite sums $\sum_{v\in V(\Delta)}|\phi(v)|^{-t}$ and $\sum_{X\in \C}|\phi(X)|^{-t}$  converge absolutely.
 \end{prop}

 \begin{proof}
 Since $\phi \in \B$, there is a finite $\phi$-attracting subtree $T$ of $\Delta$.
 Consider an arbitrary vertex $v\in V(\Delta) \setminus V(T)$. Let $v'$ be the unique vertex of distance 1 from $v$ which is closer to $T$ and suppose, without loss of generality that the edge $e$ connecting $v$ to $v'$ is labeled by $n$. Write $v=\cap_{k=1}^n X_k$ where $X_k=[v;\{\hat{k}\}] \in \C_k$ so $\bar e=X_1\cap\cdots\cap X_{n-1}$. Note that by construction, for $1 \le i <j \le n-1$, $v$ is the vertex on $[v;\{i,j\}]$ closest to $T$ and similarly, it is  also the vertex on $X_n$ closest to $T$.

Since $T$ is $\phi$-attracting, the $\phi$-induced directed edge $\vec{e}$ is directed towards $v'$.
 Therefore we have
 \begin{equation}\label{eqn:xn}
 2|x_n|\ge |x_1\cdots x_{n-1}| \Longrightarrow 2|x_n|^2\ge |x_1\cdots x_n|.
 \end{equation}

 Since $\phi\in \D$,
 $\phi(\gamma) \neq 0$ for all $\gamma \in \A$ and $\A_{\phi}(K)$ is finite for some $K>2$, so there is a constant $c=c(\phi)>0$ such that $|\phi(\gamma)|\ge c$ for all $\gamma\in \A$.
 Noticing that
 \[
 \prod_{1\le i<j\le n} \big( \prod_{m\neq i,j}x_m \big)=(x_1\cdots x_n)^{(n-1)(n-2)/2}
 \]
 and $n(n-1)/2-1=(n+1)(n-2)/2$, we have
 \begin{equation}\label{eqn:x1ton}
 |x_1\cdots x_n|^{(n-1)(n-2)/2}\ge c^{(n+1)(n-2)/2}|x_3\cdots x_n|.
 \end{equation}
 It follows from (\ref{eqn:x1ton}) and (\ref{eqn:xn}) that
 \begin{eqnarray}
 |\phi(v)|\!\!\!\!&=&\!\!\!\!|x_1\cdots x_n|\ge c^{(n+1)/(n-1)}|\phi([v;\{1,2\}])|^{2/(n-1)(n-2)}, \\
 |\phi(X_n)|\!\!\!\!&=&\!\!\!\!|x_n|\ge 2^{-1/2}c^{(n+1)/2(n-1)}|\phi([v;\{1,2\}])|^{1/(n-1)(n-2)}.
 \end{eqnarray}
 Since $v$ is the vertex in the alternating geodesic $[v;\{1,2\}]$ that is closest to $T$,
 all these alternating paths are distinct as $v$ runs over $V(\Delta)\backslash V(T)$ (with a slight adjustment to the indexing if $e$ is labeled by 1 or 2).
 Thus for any $t>0$, there is some constant $C_1=C_1(\phi)>0$ such that
 $$ \sum_{v\in V(\Delta)}|\phi(v)|^{-t} \le C_1\sum_{\gamma\in\A}|\phi(\gamma)|^{-2t/(n-1)(n-2)}< +\infty. $$
 Similarly, there is some constant $C_2=C_2(\phi)>0$ such that
 $$ \sum_{X\in \C}|\phi(X)|^{-t} \le C_2\sum_{\gamma\in\A}|\phi(\gamma)|^{-t/(n-1)(n-2)}< +\infty,$$
where the final inequality in both cases follows from Proposition \ref{prop:lowerfibconvergence}.
 This concludes the proof of Proposition \ref{prop:conv}.
 \end{proof}

 \subsection{Induced weight $\psi(\vec{e})$ on a directed edge}

 Given a $\mu$-Hurwitz map $\phi$ which takes non-zero values, for a directed edge
 $$ \vec{e}_i \leftrightarrow (\{X_1,\cdots,\hat{X}_i,\cdots,X_n\};X'_i \rightarrow X_i), \quad i=1,\cdots,n $$
(of color $i$) which points towards $v=X_1\cap\cdots\cap X_n$,
 we define a $\phi$-induced weight $\psi(\vec{e}_i)$ by
 \begin{equation}
 \psi(\vec{e}_i)=x_i \big/ \textstyle\prod_{j\neq i} x_j.
 \end{equation}
 It follows from the vertex and edge relations that
 \begin{eqnarray}
 \psi(\vec{e}_i) + \psi(-\vec{e}_i) = 1; && \\
 \sum_{i=1}^{n} \psi(\vec{e}_i) - \frac{\mu}{\phi(v)} = 1. &&
 \end{eqnarray}
 Therefore we have
 \begin{eqnarray}\label{eqn:Y}
 \psi(-\vec{e}_i) = \sum_{j\neq i}^{n} \psi(\vec{e}_j) - \frac{\mu}{\phi(v)}.
 \end{eqnarray}
 Repeated use of (\ref{eqn:Y}) then gives that, for any (finite) subtree $T$ of $\Delta$,
 \begin{equation}\label{eqn:finequality}
 \sum_{\vec{e}\in C(T)} \psi(\vec{e}) - \sum_{v \in V(T)} \frac{\mu}{\phi(v)} = 1.
 \end{equation}

 \subsection{Solving for $\psi(\vec{e})$}\label{ss:solvpsi0vece}

 Given a fixed $\phi \in \Hur_{\mu}$ which assumes non-zero values and any directed edge $\vec{e} \in \vec{E}(\Delta)$, where
 $\vec{e}\leftrightarrow (\{X_1,\cdots,\hat{X_i},\cdots,X_n\}; X'_i\rightarrow X_i)$,
 we have $\psi(\vec{e})=x_i/\prod_{m\neq i}x_m$ and $\psi(-\vec{e})=x'_i/\prod_{m\neq i}x_m$.
 By the vertex equations at $v=\bar e \cap X_i$ and $v'=\bar e \cap X'_i$,
 $\psi(\vec{e})$ and $\psi(-\vec{e})$ are the two roots of the quadratic equation in $x_i/\prod_{m\neq i}x_m$:
 \begin{equation}
 \textstyle \big(x_i\big/\prod_{m\neq i}x_m\big)^2 - \big(x_i\big/\prod_{m\neq i}x_m\big)
 +\big(\sum_{m\neq i}x_m^2-\mu\big)\big/\big(\prod_{m\neq i}x_m^2\big) =0.
 \end{equation}
 Suppose $\vec{e}=\vec{E}_{\phi}(e)$, that is, $|x'_i|\ge |x_i|$. Then $|\psi(\vec{e})|\le |\psi(-\vec{e})|$ and hence we have
 \begin{equation}\label{eqn:psivece}
 \psi(\vec{e}) = \frac12 \textstyle \Big( 1 - \sqrt{1-4\big(\sum_{m\neq i}x_m^2-\mu\big)/\prod_{m\neq i}x_m^2}\,\Big)
 \end{equation}
 where the square root is assumed to have non-negative real part.

 Let $\gamma=[v;\{i,j\}] \in \A$, where $j \neq i$. Recall that $\phi(\gamma)=\prod_{m \neq i,j} x_m$ and the square sum weight $\sigma(\gamma)=\sum_{m\neq i,j}x_m^2$.
 Then we can rewrite (\ref{eqn:psivece}) as:
 \begin{equation}\label{eqn:psivecegamma}
 \psi(\vec{e}) = \frac12 \textstyle \Big( 1 - \sqrt{1-4\big(x_j^2+\sigma(\gamma)-\mu\big)/x_j^2\phi(\gamma)^2}\,\Big)
 \end{equation}


 \subsection{Estimate of $\psi(\vec{e})-\frac{1}{2}h(\phi(\gamma))$}\label{ss:estimates}

 Recall the function $h: {\mathbb C}\backslash[-2,2] \rightarrow \mathbb C$ defined by
 $h(x)=1-\sqrt{1-4/x^2}$.
 With notation as in the previous paragraph, we have

 \begin{eqnarray*}
 \hspace{-12pt} \psi(\vec{e}) -\frac12h(\phi(\gamma))
 &\!\!\!=\!\!\!& \frac12 \bigg( \sqrt{1-\frac{4}{\phi(\gamma)^2}} - \sqrt{ 1-\frac{4\big(x_j^2+\sigma(\gamma)-\mu\big)}{x_j^2\phi(\gamma)^2} } \bigg) \\
 &\!\!\!=\!\!\!& \frac{2}{x_j^2}\frac{\sigma(\gamma)-\mu}{\phi(\gamma)^2}\bigg( \sqrt{1-\frac{4}{\phi(\gamma)^2}}+\sqrt{ 1-\frac{4\big(x_j^2+\sigma(\gamma)-\mu\big)}{x_j^2\phi(\gamma)^2} } \bigg)^{-1}
 \end{eqnarray*}
 and therefore
 \begin{eqnarray*}
 |\psi(\vec{e}) - \frac12h(\phi(\gamma))| \le \frac{2}{|x_j|^2}\frac{|\sigma(\gamma)-\mu|}{|\phi(\gamma)^2|}\bigg(\Re\sqrt{1-\frac{4}{\phi(\gamma)^2}}\bigg)^{-1},
 \end{eqnarray*}
 where $\Re(z)$ is the real part of $z\in\CC$.
 Hence, for any $\phi\in \B_{\mu}$, there is a constant $C=C(\phi)>0$ such that
 \begin{equation}\label{eqn:mainestimate}
 |\psi(\vec{e}) - \frac12 h(\phi(\gamma))| \le C|x_j|^{-2}.
 \end{equation}


 \subsection{Proofs of identities}

 \begin{thm}\label{thm:identityone}
 Suppose $\phi \in \B_{\mu}$. Then we have
 \begin{equation}\label{eqn:identityone}
 \sum_{\gamma \in \A} h(\phi(\gamma)) - \sum_{v \in V(\Delta)} \frac{\mu}{\phi(v)} = 1
 \end{equation}
 where the two infinite sums converge absolutely.

 In particular when $\mu=0$, we have
 \[
 \sum_{\gamma\in\A}  h(\phi(\gamma)) := \sum_{\gamma\in\A} \bigg( 1- \sqrt{1-\frac{4}{\phi(\gamma)^2}} \bigg) = 1.
 \]
 \end{thm}

 \begin{proof}
 Since $\phi \in \B_{\mu}$,
 there is a nonempty $\phi$-attracting subtree $T \subset \Delta$, that is, for each edge $e \in E(\Delta)$ not contained in $T$,
 the $\phi$-induced directed edge $\vec{E}_{\phi}(e)$ is directed towards $T$.
 Let $T_m \subset\Delta$, $m \ge 0$ be the subtree spanned by those vertices of $\Delta$ that are of distance not exceeding
 $m$ from $T$. 
 By (\ref{eqn:finequality}), we have, for all $m \ge 0$,
 \begin{equation}\label{eqn:finequalitym}
 \sum_{\vec{e}\in C(T_m)} \psi(\vec{e}) - \sum_{v \in V(T_m)} \frac{\mu}{\phi(v)} = 1.
 \end{equation}

 Let $\A_m \subset\A$, $m \ge 0$ be the set of those elements of $\A$ that are of distance not exceeding $m$ from $T$.
 Then $\A=\bigcup_{m=0}^{\infty}\A_{m}=\A_0\sqcup \bigsqcup_{m=0}^{\infty}(\A_{m+1}\backslash\A_{m})$.
 Similarly, let $\C(m)\subset \C$, $m \ge 0$ be the set of those elements of $\C$ that are of distance not exceeding $m$ from $T$;
 then we have
 $$ \C=\textstyle\bigcup_{m=0}^{\infty}\C(m)=\C(0)\sqcup \bigsqcup_{m=0}^{\infty}(\C(m+1)\backslash \C(m)). $$

 We claim that, as $m \rightarrow \infty$,
 \begin{equation}\label{eqn:toprove}
 \sum_{\gamma\in\A_m}h(\phi(\gamma)) - \sum_{v \in V(T_{m+1})} \frac{\mu}{\phi(v)} - 1 \rightarrow 0.
 \end{equation}
 The desired identity (\ref{eqn:identityone}) then follows by passing to limits.

 Given an arbitrary $\vec{e}\in C(T_{m+1})$, $m\ge 0$, suppose the underlying edge $e$ is of color $i$ and
 let $v=X_1\cap\cdots\cap X_n$ be the head vertex of $\vec{e}$, where $X_1,\cdots, X_n \in \C$ are indexed by colors.
 Let $e^*$ be the edge incident to $v$ that is closer to $T$ than $e$. Suppose $e^*$ is of color $j$. Then $j\neq i$.
 Let $\gamma=[v;\{i,j\}]\in\A$. Then $\gamma \in \A_{m}$.
 In this way, to each $\vec{e}\in C(T_{m+1})$ is associated a unique element $\gamma\in\A_{m}$.
 Note that each element $\gamma\in\A_{m}$ contains exactly two edges in $C(T_{m+1})$.
 Hence each $\gamma\in\A_{m}$ is associated to exactly two edges, say $\vec{e}$ and $\vec{e}\,'$, in $C(T_{m+1})$.
 On the other hand, to each $\vec{e} \in C(T_{m+1})$ is associated $X_j\in \C_j \bigcap (\C(m+1)\backslash \C(m))$,
 and conversely, each $X \in \C(m+1)\backslash \C(m)$ is associated to exactly $n-1$ edges in $C(T_{m+1})$.

 By (\ref{eqn:finequalitym}) and the above associations, the left side of (\ref{eqn:toprove}) equals
 \begin{eqnarray*}
 & &\sum_{\gamma\in\A_m} h(\phi(\gamma)) - \sum_{\vec{e}\in C(T_{m+1})} \psi(\vec{e}) \\
 &=&\sum_{\gamma\in\A_m} ([\frac12h(\phi(\gamma))-\psi(\vec{e})] + [\frac12 h(\phi(\gamma))-\psi(\vec{e}\,')]) \\
 &=&\sum_{\vec{e}\in C(T_{m+1})} (\frac12 h(\phi(\gamma)) - \psi(\vec{e})).
 \end{eqnarray*}

 Since $\phi\in \B_{\mu}$, by the discussion in \S \ref{ss:estimates}, there is a constant $C=C(\phi)>0$ such that
 $|\frac12 h(\phi(\gamma))-\psi(\vec{e})| \le C|\phi(X_j)|^{-2}$.
 Hence the modulus of left side of (\ref{eqn:toprove}) equals
 \begin{eqnarray*}
 &   & \Big|\sum_{\gamma\in\A_m}h(\phi(\gamma)) - \sum_{\vec{e}\in C(T_{m+1})} \psi(\vec{e})\;\Big| \\
 &\le& \sum_{\vec{e}\in C(T_{m+1})}\big|\frac12 h(\phi(\gamma)) - \psi(\vec{e})\big| \\
 &\le& \sum_{X\in \C(m+1)\backslash \C(m)} (n-1) C |\phi(X)|^{-2} \longrightarrow 0 \quad\quad \text{as} \;\  m \rightarrow \infty
 \end{eqnarray*}
 since the infinite sum $\sum_{X\in \C} |\phi(X)|^{-2}$ converges by Proposition \ref{prop:conv}.
 This proves the claim and thus finishes the proof of Theorem \ref{thm:identityone}.
 \end{proof}


 Theorem \ref{thm:identityone} is the special case of Theorem \ref{thm:identity} when $\mu=0$. We recall the general version  when $\mu \neq 0$:

 Suppose $\phi \in \B_{\mu}$. Then we have
 \begin{equation}\label{eqn:omegaid}
 \sum_{\gamma\in\A} \mathfrak{h}_{\mu}(\gamma): = \sum_{\gamma\in\A} \bigg( 1-\bigg(1+\frac{2 \mu}{n(n-1)(\sigma(\gamma)-\mu)} \bigg) \sqrt{1-\frac{4}{\phi(\gamma)^2}} \bigg) = 1,
 \end{equation}
 where the infinite sum converges absolutely.

 \begin{proof} [Proof of Theorem \ref{thm:identity}]
 We shall obtain the desired identity (\ref{eqn:omegaid}) by rewriting the second infinite sum on the left side of identity (\ref{eqn:identityone}) as follows.

 To each pair $(v,\gamma)$ where $\gamma\in\A$ and $v\in V(\gamma)$, we associate the value $\frac{2}{n(n-1)}\frac{\mu}{\phi(v)}$.
 Note that passing through any vertex $v\in V(\Delta)$, there are exactly $n(n-1)/2$ elements $[v;\{i,j\}]\in \A$, $1\le i<j \le n$.
 Therefore we have
 $$ \sum_{v \in V(\Delta)} \frac{\mu}{\phi(v)} = \frac{2}{n(n-1)}\sum_{\gamma\in\A}\sum_{v\in V(\gamma)}\frac{\mu}{\phi(v)}. $$
 By a summation identity (Proposition 4.1 in \cite{hu-tan-zhang2014mrl}), we have
 $$ \sum_{v\in V(\gamma)}\frac{\mu}{\phi(v)} = \frac{\mu}{\sigma(\gamma)-\mu}\sqrt{1-\frac{4}{\phi(\gamma)^2}}, $$
 from which (\ref{eqn:omegaid}) follows immediately. This completes the proof of Theorem \ref{thm:identity}.
 \end{proof}

%

 \subsection{Variations of identities} Theorem \ref{thm:identity} is a special case (with $q_{ij}=2/n(n-1)$) of the following theorem.

 \begin{thm}\label{thm:omegaidgen}
 Suppose $\phi \in \B_{\mu}$. Choose arbitrary $q_{ij}\in \mathbb{C}$ with $\sum_{1\le i<j \le n} q_{ij}=1$, then we have
 \begin{equation}\label{eqn:omegaidgen}
 \sum_{1\le i<j \le n}\sum_{\gamma\in\A_{ij}} \mathfrak{h}_{\mu}(\gamma,q_{ij}): = \sum_{1\le i<j \le n} \sum_{\gamma\in \A_{ij}} \bigg( 1-\bigg(1+q_{ij} \frac{\mu}{\sigma(\gamma)-\mu} \bigg) \sqrt{1-\frac{4}{\phi(\gamma)^2}} \bigg) = 1,
 \end{equation}
 where the infinite sums converge absolutely.
 \end{thm}

 \begin{proof}
 The proof is essentially the same as the proof of Theorem \ref{thm:identity}. Instead of assigning the same weight $2/n(n-1)$ to all $\gamma \in \A$, we assign the weight $q_{ij}$ to $\gamma \in \A_{ij}$.
 \end{proof}

 We also have the following relative version of the identities, in the case where $\phi$ may not satisfy the Bowditch conditions (B1) and (B2), but they are satisfied for a subset $ \A^{0-}(\vec{e}) \subset \A$ for some $\phi$-directed edge $\vec{e}\in \vec{E}(\Delta)$. Essentially the same proof yields the following:

 \begin{prop}
 For $\phi \in \Hur_{\mu}$ and $e \in E(\Delta)$, let $\vec{e} = \vec{E}_{\phi}(e)$. Suppose that \[\A^{0-}(\vec{e}) \bigcap \phi^{-1}[-2,2] = \emptyset,\]and for some $K >2$,
 \[\A^{-}(\vec{e}) \bigcap \A_{\phi}(K)\] is finite. Then
 \begin{equation}
 \label{eq:edgeidentity}
 \psi(\vec{e}) = \sum_{\gamma \in \A^0(e)} \frac12 \mathfrak{h}_{\mu}(\gamma) + \sum_{\gamma' \in \A^-(\vec{e})}  \mathfrak{h}_{\mu}(\gamma'),
 \end{equation}
 where the infinite sum on the right side converges absolutely.

 In particular when $\mu=0$, we have
 \[
 \psi(\vec{e}) = \sum_{\gamma \in \A^0(e)} \frac12 h(\phi(\gamma)) + \sum_{\gamma' \in \A^-(\vec{e})}  h(\phi(\gamma')).
 \]
 \end{prop}

 \begin{rmk}
 With the above proposition and following the computations in the proof of Proposition 2.9 of \cite{tan-wong-zhang2006gd}, we can show that \eqref{eq:edgeidentity} still holds when $\phi(\gamma) = \pm 2$ for a finite number of $\gamma \in \A^{0-}(\vec{e})$, thus the identity \eqref{eqn:omegaid} in Theorem \ref{thm:identity} holds as well in such conditions.
 \end{rmk}

 \section{Further directions}\label{s:conclusion}
 There are many interesting questions and directions to be explored. We list some below:
 \begin{enumerate}
\item Foremost perhaps is the question of whether the sets $\D$ defined, and the identities derived here have interesting geometric interpretations in terms for example of hyperbolic surfaces or other geometric objects. As noted in the introduction, the case $n=3$ is well studied and corresponds to the character variety of $F_2$, and the original McShane identity, the case $n=4$ also has the interesting interpretation in terms of the Fricke space of the thrice-punctured projective plane from the work of Huang-Norbury.

\item The geometry of the regular subtrees $T^{|k|}(\Delta)$ for $1\le k<n-1$ and the extended Hurwitz map on these sets is an interesting topic. It would be interesting to see if the geometric picture of the Siepinski simplices we discussed for $\A=T^{|2|}(\Delta)$ generalizes to these objects, and if the extended Hurwitz maps have interesting properties which characterize them, for example, edge relations etc.

\item In this paper, we described a non-empty open subset $\D$ of $\CC^n$ on which $\Gamma_n$ acts properly discontinuously. A natural question is if this is the largest open subset with this property. Other basic properties of this set are also unknown, for example if this set is connected, if the boundary is always fractal-like, etc. Restriction of the action to  $\RR^n$ or other subfields is also interesting. 	One could also look at the subset of Hurwitz maps which are fixed by some element of $\Gamma_n$, that is, for which $\phi_{\mathbfit{a}}=\phi_{g(\mathbfit{a})}$ for some $g\in \Gamma_n$. In this case, the Bowditch conditions will not be satisfied but we can define relative Bowditch conditions similar to \cite{tan-wong-zhang2008advm}.

\item The dynamics of the action of $\Gamma_n$ on the complement of the closure of $\D$ seems quite mysterious and worth exploring. Of interest for example would be if there exists some easily defined subset on which the action of $\Gamma_n$ is ergodic.

\item Finally, the methods described in this paper are fairly general and it would be interesting to see if they can be applied to for example the study of Apollonian circle/sphere packings where a similar group action is observed.
\end{enumerate}


\vskip 18pt

\section*{Appendix: Polynomial automorphisms}\label{s:appendix}

\noindent Horowitz in \cite{horowitz1975tams} studied the group of polynomial automorphisms of $\mathbb{C}^3$ preserving the polynomial
\[
x^2+y^2+z^2-xyz.
\]
Here we extend the study and give the proof of Theorem \ref{thm:polyauto}.

\begin{proof}
It is clear that the elements of $\Gamma_n$,  $\Upsilon_n$ and $S_n$ are polynomial automorphisms of $\CC^n$ preserving the  Markoff-Hurwitz polynomial $H(\mathbfit{x})$, thus the group generated by these elements is a subgroup of $\Gamma_n^*$. To complete the proof, we need to show the other inclusion. Let
\[
x_1 \to R_1, x_2 \to R_2, \cdots, x_n \to R_n
\]
be an arbitrary element $\Theta \in \Gamma_n^{\ast}$ where $R_1, R_2, \cdots, R_n$ are polynomials in $x_1, x_2,\cdots$, $x_n$ of degree $r_1,r_2,\cdots,r_n$, respectively. Since $\Theta$ keeps invariant the polynomial $H(\mathbfit{x})$, we have
\begin{equation}
\label{eq:rightleftpoly}
R_1^2+R_2^2+\cdots +R_n^2-\prod_{i=1}^n R_i = x_1^2+x_2^2 + \cdots+x_n^2-\prod_{i=1}^n x_i.
\end{equation}
We observe that the degree of the right-hand side polynomial is $n$ and the degree of the left hand side polynomial is not great than $\max\{2r_1, 2r_2, \cdots, 2r_n, \sum_{i=1}^n r_i\}$. In the following, we follow the method in \cite{horowitz1975tams} and decompose polynomials into sums of homogeneous polynomials, that is,
\begin{gather*}
R_1=R_1^{(r_1)}+R_1^{(r_1-1)}+\cdots +R_1^{(0)},R_1^{(r_1)}\neq 0,R_1^{(0)} \in \mathbb{C},\\
R_2=R_2^{(r_2)}+R_2^{(r_2-1)}+\cdots +R_2^{(0)},R_2^{(r_2)}\neq 0,R_2^{(0)} \in \mathbb{C},\\
\cdots,\\
R_n=R_n^{(r_n)}+R_n^{(r_n-1)}+\cdots +R_n^{(0)},R_n^{(r_n)}\neq 0,R_n^{(0)} \in \mathbb{C},
\end{gather*}
where $R_j^{(k)}$ is the sum of the monomials of degree $k$ in $R_j$, that is, the homogeneous polynomial of degree $k$ in $R_j$.

\vspace{10pt}

We prove by induction on the highest degree of $\Theta$. Note that by applying permutation automorphisms, we can always assume $r_1 \le r_2 \le \cdots \le r_n$. Furthermore, we know that if one of the polynomials $R_1, R_2, \cdots, R_n$ is of degree zero, i.e just a constant, then $\Theta$ is not an automorphism of $\mathbb{C}^n$. Thus we have $r_i \ge 1$ for all $i$.

\vspace{5pt}

First suppose $r_1=r_2=\cdots=r_n=1$, i.e. $\Theta$ is a linear automorphism. Then \eqref{eq:rightleftpoly} implies $\prod_{i=1}^n R_i^{(1)}=\prod_{i=1}^n x_i$ by comparing the highest terms on both sides. Thus by unique factorization, the polynomials $R_1^{(1)},R_2^{(1)},\cdots, R_n^{(1)}$ are $c_1 x_1, c_2 x_2, \cdots, c_n x_n$ with $c_i \in \mathbb{C}\backslash \{0\}$ and $\prod_{i=1}^n c_i = 1$.  By composing with a permutation, we can assume that
\[
R_1^{(1)}=c_1 x_1, R_2^{(1)}=c_2 x_2, \cdots, R_n^{(1)}=c_n x_n.
\]
Comparing the degree $(n-1)$-terms of \eqref{eq:rightleftpoly} implies $R_1^{(0)} = R_2^{(0)}= \cdots =R_n^{(0)}=0$. Then $c_i =\pm 1$ and $\prod_{i=1}^n c_i = 1$, that is the automorphism is a sign-change automorphism. Therefore we have proved that the group $\Lambda$ of linear automorphisms of $\CC^n$ preserving $H(\mathbfit{x})$ is generated by elements of $\Upsilon_n$ and $S_n$. It is easy to check that $\Upsilon_n \lhd \Lambda$ and $\Upsilon_n \bigcap S_n$ is trivial, so $\Lambda = \Upsilon_n \rtimes S_n$. It is also clearly finite.

\vspace{5pt}

Now we may assume $r_n >1$. Here we claim that $r_n=\sum_{i=1}^{n-1} r_i$.  Otherwise, if $r_n > \sum_{i=1}^{n-1} r_i \ge n-1$, then $(R_n^{(r_n)})^2$ is the only non zero homogeneous polynomial of degree $2 r_n > n$ in the left-hand side of \eqref{eq:rightleftpoly}, which implies that \eqref{eq:rightleftpoly} cannot be true; or else, if $r_n < \sum_{i=1}^{n-1} r_i $, then $\prod_{i=1}^n R_i^{(r_i)}$ is the only non zero homogeneous polynomial of degree $\sum_{i=1}^{n} r_i > n$ in the left-hand side of \eqref{eq:rightleftpoly}, which also implies that \eqref{eq:rightleftpoly} cannot be true either.

\vspace{5pt}

Therefore, $r_n=\sum_{i=1}^{n-1} r_i$ and the terms of degree greater than $n$ in the left-hand side of \eqref{eq:rightleftpoly} must  cancel. The sum of terms of highest degree on the left-hand side of \eqref{eq:rightleftpoly} is $(R_n^{(r_n)})^2-\prod_{i=1}^n R_i^{(r_i)}$, thus $R_n^{(r_n)}=\prod_{i=1}^{n-1} R_i^{(r_i)}$. Now consider the following polynomial automorphism of $\mathbb{C}^n$
\[
x_1 \to R_1'=R_1, x_2 \to R_2'= R_2, \cdots, x_{n-1}\to R_{n-1}'= R_{n-1}, x_n \to R_n'= \prod_{i=1}^{n-1} R_i -R_{n},
\]
which is the composition  $b_n\circ \Theta$. We can check that the maximal degree of the above polynomials $\{R_i'\}_{i=1}^n$ is strictly less than $r_n$. Therefore by induction, $b_n\circ \Theta$, and hence $\Theta$ are both  generated by elements in $\Gamma_n$, $\Upsilon_n$ and $S_n$.

\vspace{5pt}

It is straightforward to show that $\Gamma_n \lhd \Gamma_n^{\ast}$ and $\Gamma_n \bigcap \Lambda$ is trivial, hence $\Gamma_n^{\ast} = \Gamma_n \rtimes \Lambda$, which completes the proof.
\end{proof}

\end{document}